\documentclass[letterpaper, 11pt,  reqno]{amsart}
\usepackage[margin=1.2in,marginparwidth=1.5cm, marginparsep=0.5cm]{geometry}

\usepackage{amsmath,amssymb,amscd,amsthm,amsxtra, esint, xcolor,mathtools,enumerate}
\usepackage{bbm}
\usepackage[implicit=true]{hyperref}
\usepackage{mathtools}
\allowdisplaybreaks[2]

\usepackage{color}

\newtheorem{theorem}{Theorem} [section]

\newtheorem{lemma}[theorem]{Lemma}
\newtheorem{proposition}[theorem]{Proposition}
\newtheorem{remark}[theorem]{Remark}

\newtheorem{corollary}[theorem]{Corollary}

\newtheorem{assumption}{Assumption}

\newcommand{\I}{\hspace{0.5mm}\text{I}\hspace{0.5mm}}
\newcommand{\II}{\text{I \hspace{-2.8mm} I} }

\numberwithin{equation}{section}
\numberwithin{theorem}{section}

\makeatletter
\@namedef{subjclassname@2020}{\textup{2020} Mathematics Subject Classification}
\makeatother

\begin{document}
\baselineskip = 14pt

\title[Exact variations for nonlinear stochastic heat equation]
{Temporal quartic variation for nonlinear stochastic heat equations with piecewise constant coefficients}

\author[Y. Li, Y. Hu, L. Yan and H. Shu]
{Yongkang Li, Yaozhong Hu, Litan Yan$^*$ and Huisheng Shu}

\thanks{Y. Li is supported by the Fundamental Research Funds for the Central Universities of China (CUSF-DH-T-2024066). H. Shu is supported by the National Natural Science Foundation of China (No. 62073071). L. Yan is supported by the National Natural Science Foundation of China (No. 11971101) and the Natural Science Foundation of Shanghai Municipality (24ZR1402900).  Y. Hu is supported by   
 NSERC Discovery grant RGPIN
2024-05941  and a centennial  fund of University of Alberta.}

\address{Yongkang Li\\
	School of Mathematics and Statistics, Donghua University, 2999 North Renmin Rd., Songjiang, Shanghai 201620, P.R. China
	}
\address{
	Department of Mathematical and statistical sciences, University of Alberta, Edmonton, AB, T6G 2G1, Canada}
\email{yongkangli@mail.dhu.edu.cn}

\address{Yaozhong Hu\\
	Department of Mathematical and Statistical Sciences, University of Alberta, Edmonton, AB T6G 2G1, Canada}
\email{yaozhong@ualberta.ca}

\address{
	Huisheng Shu\\
	School of Mathematics and Statistics, Donghua University, 2999 North Renmin Rd., Songjiang, Shanghai 201620, P.R. China
}

\email{hsshu@dhu.edu.cn}

\address{
	Litan Yan\\
	School of Mathematics and Statistics, Donghua University, 2999 North Renmin Rd., Songjiang, Shanghai 201620, P.R. China
}

\email{litan-yan@hotmail.com}



\begin{abstract}  
We consider a stochastic partial differential equation with piecewise constant coefficients driven by multiplicative space-time white noise. We establish the existence and uniqueness of its mild solution in the sense of Walsh and investigate the asymptotic behavior of its temporal quartic variation. As an application, we construct a consistent estimator based on the resulting quartic variation formula.
\end{abstract}

\date{\today}
\maketitle
%

%

\baselineskip = 14pt

\section{Introduction}
Recently, Zili and Zougar \cite{MR3988829} introduced the following stochastic partial differential equation
 \begin{equation}\label{eq:main equation}
 \left\{
 \begin{aligned}
 &\frac{\partial}{\partial t}u(t,x)=\mathcal{L}\, u(t,x)+\sigma(u(t,x))\dot{W},\quad t\geq0,x\in\mathbb{R}\\
 &u(0,x)=0
 \end{aligned}
 \right.
 \end{equation}
where the operator $ \mathcal{L}$ is defined as
\begin{equation}\label{eq:L}
\mathcal{L}=\frac{1}{2\rho(x)}\frac{d}{dx}\left(\rho(x)A(x)\frac{d}{dx}\right),
\end{equation}
with
\begin{equation}\label{eq:A(x)}
A(x)=a_1\mathbbm{1}_{\{x\leq 0\}}+a_2\mathbbm{1}_{\{x>0\}}, \\
\rho(x)=\rho_1\mathbbm{1}_{\{x\leq 0\}}+\rho_2\mathbbm{1}_{\{x>0\}},
\end{equation}
and $ a_i, \rho_i (i=1,2)$ are positive constants. Also, the $ \dot{W} $ denotes the formal derivative of $W$. More precisely, $ W $ is a centered Gaussian stochastic field $ W=\{W(t, C); t\in[0, T], C\in \mathcal{B}_b(\mathbb{R})\} $ with covariance
\[\mathbb{E}[W(t, B)W(s, C)]=(t\wedge s)\lambda(B\cap C),\]
where $ \lambda $ is the Lebesgue measure and  $ \mathcal{B}_b(\mathbb{R}) $ is the set of $ \lambda- $bounded Borel subsets of $ \mathbb{R} $. Thus, \(W\) is interpreted as a space--time white noise in the
sense of Walsh's martingale-measure framework. We will assume that the coefficient $ \sigma:\mathbb{R}\rightarrow\mathbb{R} $ is Lipschitz continuous, i.e., there exists $M>0$ such that, for all \(x,y\in\mathbb R\),
\begin{equation}\label{eq:sigma}
|\sigma(x)-\sigma(y)|\leq M|x-y|.
\end{equation}
Additionally, the solution to equation (\ref{eq:main equation}) is a random field $ \{u(t,x), t\geq0, x\in\mathbb{R}\} $. An interesting fact is that in our case, the coefficient $ A $ in (\ref{eq:A(x)}) is a piecewise constant with one point of discontinuity which reflects the heterogeneity of the medium consisting of two kinds of materials in which the process under study propagates. For this reason, this equation characterizes the phenomenon of heat propagation within a medium composed of two distinct materials, subject to stochastic disturbances. Moreover, such a phenomenon is frequently observed in numerous practical scenarios; for example, see Cantrell and Cosner~\cite{CANTRELL1999189},  \'Etor\'e~\cite{MR2217816}, Gaveau \emph{et al.}~\cite{MR917463}, Lejay~\cite{MR2105066}, Nicaise~\cite{MR839024}. And this is one reason why we are particularly interested in this equation. From a mathematical perspective, heat propagation is represented by heat equations involving divergence-form elliptic operators with discontinuous coefficients. 

Equation (\ref{eq:main equation}) can be seen as a stochastic counterpart of the deterministic partial differential equation
\begin{equation}\label{eq:pde}
\frac{\partial u(t,x)}{\partial t}=\mathcal{L} u(t,x).
\end{equation}
An explicit expression of the fundamental solution of equation (\ref{eq:pde}) was obtained in Chen and Zili \cite{MR3296333}.

Equation (\ref{eq:main equation}) can also be viewed as a natural generalization of the stochastic heat equation driven by space-time white noise. This research has been extensively explored in a previous monograph (see, Tudor~\cite{MR4544909}). Usually, the operator $\mathcal{L}$ was taken as a Laplacian or fractional Laplacian operator. This provides additional motivation for examining the solution to this type of equation. 

Stochastic partial differential equations(SPDEs) associated with the operator $ \mathcal{L} $ defined in (\ref{eq:L}) have been investigated in Zili and Zougar  ~\cite{MR3988829,MR4421358,MR4028081,MR3992994}, and Mishura\emph{et al.}~\cite{MR4192903}. While many of these references concern the case of SPDEs with additive Gaussian noise, there exist few works which treat the case of a stochastic diffusion coefficient $ \sigma $ in (\ref{eq:main equation}). We may mention that another type of (\ref{eq:main equation}) involving $\mathcal{L} $ with $\sigma(u(t,x))$ formed as an affine function, i.e. $\sigma(x):=h_1x+h_2$, $h_1, h_2\in\mathbb{R}$, driven by an infinite-dimensional fractional Brownian motion has been studied in Mishura\emph{et al.}~\cite{MR4192903}. Among these works, M. Zili and E. Zougar~\cite{MR3992994} expanded the quartic variations in time and the quadratic variations in space of the solution to the equation (\ref{eq:main equation}) with $ \sigma\equiv 1$. They also deduced the estimators of $a_1$ and $a_2$ appearing in (\ref{eq:A(x)}). Zili and Zougar \cite{MR4028081} proved that the sequence of recentered and renormalized spatial quadratic variations of the solution to (\ref{eq:main equation}) satisfies an almost sure central limit theorem by using the Malliavin-Stein method. Besides, they gave some useful moment bound concerning the increment of the solution and the kernel. 

The exact \(p\)-variation has become an effective tool for parameter identification for SPDEs.  
\cite{MR2321899}  studied the exact quadratic variation in space and quartic variation in time for the one-dimensional stochastic heat equation driven by multiplicative space-time white noise. Subsequent works extended this framework to fractional stochastic heat equations with fractional Laplacian operators, corresponding to \(\mathcal L:=-(-\Delta)^{\alpha/2}, \alpha\in(1,2]\); see \cite{gamain2023exact,li2026exactspace,li2026exacttime,olivera2025temporal}. Exact spatial variation problems for stochastic heat equations with piecewise constant coefficients have also been investigated in \cite{li2025spatial}.
However, the temporal variations in this multiplicative setting have not yet been investigated.

In this work, we further develop the analysis of SPDEs of the form (\ref{eq:main equation}). We establish the temporal quartic variation with an explicit \(L^1\) convergence rate for the nonlinear divergence-form model considered here and discuss its application to parameter estimation. To establish the proofs, we employ various integration and analytical tools, making essential use of the explicit fundamental solution of (\ref{eq:main equation}) obtained by Chen and Zili~\cite{MR3296333}.  For the main techniques we refer to the similar arguments in Posp\'i\v{s}il and Tribe~\cite{MR2321899}.  

Our work is organized as follows: In the second section, we give some useful estimates for the fundamental solution of (\ref{eq:pde}) and prove the existence and uniqueness of the solution (\ref{eq:solution}) to the nonlinear stochastic partial differential equation (\ref{eq:main equation}). The last section is devoted to giving the explicit expression of the temporal quartic variations. Such an expansion allows us to derive an estimation procedure for the parameters $a_1$ and $a_2$.

\section{Preliminaries}
\subsection{The fundamental solution and Walsh stochastic integrals}
The solution to (\ref{eq:main equation}) is defined in the mild sense, that is, it can be written as a Dalang-Walsh integral concerning the noise $W$ by
\begin{equation}\label{eq:solution}
u(t,x)=\int_0^t\int_{\mathbb{R}}\sigma(u(s,y))G(t-s,x,y)\,W(ds,dy)\,, \quad t\in[0,T],\, x\in\mathbb{R}
\end{equation}
where \(G(t-s,x,y)\) denotes the fundamental solution of \eqref{eq:pde};
see \cite{arXiv:2502.15351} for its positivity. Its explicit expression is
given by
\begin{multline}\label{eq:G}
G(t-s,x,y)
 =   \Bigg[ \frac{1}{\sqrt{2 \pi (t-s)}}
 \left( \frac{ \mathbbm{1}_{ \{ y \le 0 \} } }{\sqrt{a_1}} + \frac{ \mathbbm{1}_{ \{y>0\} } }{\sqrt{a_2}} \right)  \left\{ \exp \left( - \frac{(f(x) - f(y))^2}
{2 (t-s)}\right) \right. \\
  +  \left.
\beta \, {\rm sign} (y)\, \exp
\left( - \frac{(\mid f(x) \mid  + \mid f(y) \mid)^2}{2(t-s)}\right) \right\} \Bigg] \mathbbm{1}_{\{s < t\}}\,,
\end{multline}
with \[f(y)= \frac{y}{\sqrt{a_1}}{\mathbbm 1}_{\{y\leq 0\}}+\frac{y}{\sqrt{a_2}} {\mathbbm 1}_{\{y>0\}}\,,
\quad
\alpha  = 1- \frac{\rho_1 a_1}{\rho_2a_2}\,,
\quad\text{and}\quad
\beta= \frac{\sqrt{a_1}+ \sqrt{a_2} (\alpha -1)}{\sqrt{a_1}-\sqrt{a_2} (\alpha -1)}   .\] 
For the proof, see e.g., Chen and Zili~\cite{MR3296333}.

We briefly recall the Walsh stochastic integral. For \(t\geq0\), let
\(\mathcal F_t\) denote the sigma-algebra generated by
\[
\{W(s,A):0\leq s\leq t,\ A\in\mathcal B_b(\mathbb R)\}.
\]
In what follows, all random fields are assumed to be jointly measurable and
adapted to \(\{\mathcal F_t\}_{t\geq0}\). When used as integrands in Walsh
stochastic integrals, they are understood to be predictable and to satisfy
\[
\mathbb E\left[
\int_0^\infty\int_{\mathbb R} X(s,y)^2\,dy\,ds
\right]<\infty.
\]
Moreover, we will need the following formula which follows from (\ref{eq:solution}) and the It\^o-type isometry associated with the Dalang-Walsh integral (the details can be seen, for example, in Dalang~\cite{MR1684157} and Walsh~\cite{MR876085} and this approach is well-suited to the study of trajectory space-time properties of the solution)
\[\mathbb{E}[u(t,x)\,u(s,y)]=\int_{0}^{s\wedge t}\int_{\mathbb{R}}G(t-r,x,z)\,G(s-r,y,z)\,\mathbb{E}[\sigma(u(r,z))^2]\,dz\,dr\,,\]
for any $s,t>0$ and $x,y\in\mathbb{R}$.

We shall repeatedly use the Burkholder--Davis--Gundy inequality for Walsh
stochastic integrals: for every random field \(X\) satisfying
the usual integrability condition and for every \(p\ge2\),
\[
\left\|
\int_0^t\int_{\mathbb R}X(r,z)\,W(dr,dz)
\right\|_p^2
\le
C_p
\int_0^t\int_{\mathbb R}
\|X(r,z)\|_p^2\,dz\,dr .
\]
In the case \(p=2\), we use the It\^o isometry.
\subsection{Kernel estimates}
To address our issue, we will need some properties with respect to the fundamental solution $G(t-s,x,y)$.  For simplicity, throughout the paper, \(C\) denotes a finite positive constant whose
value may change from line to line. It may depend on the fixed model
parameters \(a_1,a_2,\rho_1,\rho_2\), as well as on the Lipschitz
constant \(M\) of \(\sigma\) and on \(|\sigma(0)|\). Additional
dependence on \(p,T,\varepsilon\), or a fixed spatial point \(x\), is
indicated by subscripts. Unless otherwise stated, the constants are
independent of \(n,i,j,\delta\) and of the free space--time variables.
\begin{lemma}\label{le:Gleq}
For every $0\leq s<t\leq T$ and $x, y\in\mathbb{R}$, we have
\[|G(t-s,x,y)|\leq C\frac{1}{\sqrt{t-s}}\exp\left(-\frac{(f(x)-f(y))^2}{2(t-s)}\right)\] 
with $C=\frac{1+|\beta|}{\sqrt{2\pi}}(\frac{1}{\sqrt{a_1}}+\frac{1}{\sqrt{a_2}})$.
\end{lemma}
\begin{proof}
See Mishura~\emph{et al.} \cite{MR4192903}.
\end{proof}
\begin{lemma}\label{le:intGleq}
For every $0\leq s<t\leq T$ and $x, y\in\mathbb{R}$, we have
\[\max \left\{ \int_{\mathbb{R}} |G(t-s,x,z)|\,dz, \int_{\mathbb{R}} |G(t-s,y,z)|\,dz \right\} \le C,\]
where $C=(\frac{1}{\sqrt{a_1}}+\frac{1}{\sqrt{a_2}})(1+|\beta|)\max(\sqrt{a_1},\sqrt{a_2}).$
\end{lemma}
\begin{proof}
See Mishura~\emph{et al.} \cite{MR4192903}.
\end{proof}

\begin{lemma}\label{le:quasi-semigroup}
For every $0\leq s<t\leq T$ and $x\in\mathbb{R}$, there exists a constant \(C>0\) such that
\[\int_{\mathbb{R}}G(t-s,x,y)^2dy\leq C\frac{1}{2\sqrt{\pi}}\frac{1}{\sqrt{t-s}}.\]
\end{lemma}
\begin{proof}
According to Lemma \ref{le:Gleq}, we have
\begin{multline*}
\int_{\mathbb{R}}G(t-s,x,y)^2dy\leq C^2\int_0^{+\infty}\left| \frac{1}{\sqrt{2\pi(t-s)}}
\exp\left(-\frac{\left(f(x)-\frac{y}{\sqrt{a_2}}\right)^2}{2(t-s)}\right) \right|^2dy\\
+C^2\int_{-\infty}^0\left| \frac{1}{\sqrt{2\pi(t-s)}}
\exp\left(-\frac{\left(f(x)-\frac{y}{\sqrt{a_1}}\right)^2}{2(t-s)}\right) \right|^2dy
\end{multline*}
By changing the variables
\[y_1=f(x)-\frac{y}{\sqrt{a_2}},\, y_2=f(x)-\frac{y}{\sqrt{a_1}},\]
we obtain
\begin{align*}
\int_{\mathbb{R}}G(t-s,x,y)^2dy&\leq C^2\sqrt{a_2}\int_{-\infty}^{f(x)}\left| \frac{1}{\sqrt{2\pi(t-s)}}\exp\left(-\frac{y^2}{2(t-s)}\right) \right|^2dy\\
&\qquad\qquad+C^2\sqrt{a_1}\int_{f(x)}^{+\infty}\left| \frac{1}{\sqrt{2\pi(t-s)}}\exp\left(-\frac{y^2}{2(t-s)}\right) \right|^2dy\\
&\leq (\sqrt{a_1}\vee\sqrt{a_2}) C^2\int_\mathbb{R}\left| \frac{1}{\sqrt{2\pi(t-s)}}\exp\left(-\frac{y^2}{2(t-s)}\right) \right|^2dy\\
&:=C^2\int_{\mathbb{R}}p_{t-s}^2(y)dy\,,
\end{align*}
where $p_{t}(y)=(2\pi t)^{-1/2}\exp(-y^2/2t)$ denotes the fundamental solution associated to the standard heat equation.
Then we apply the Parseval-Plancherel identity by recalling that the Fourier transform of $p_t(\cdot)$ is
\[\mathcal{F}p_t(\cdot)(\xi)=e^{-\frac{t\xi^2}{2}}\,,\quad\xi\in\mathbb{R}\,,\,t>0.\]
Hence, it follows that
\begin{equation*}
\int_{\mathbb{R}}G(t-s,x,y)^2dy\leq C^2\int_{\mathbb{R}}\frac{1}{2\pi}e^{-(t-s)\xi^2}d\xi=C^2\frac{1}{2\sqrt{\pi}}\frac{1}{\sqrt{t-s}}.
\end{equation*}

\end{proof}

By using similar arguments from Zili and Zougar~\cite{MR4028081}, we get the following estimate. The proof is postponed to Appendix \ref{appendix}.
\begin{lemma}\label{le:G-G2}
	Let $G(t,x,y)$ be the heat kernel defined by (\ref{eq:G}). Then for every $s,t\in[0,T], x\in\mathbb{R}$ with $s< t$, and for every $h>0$, we have
	\begin{equation}
		\int_0^s\int_{\mathbb{R}}\left(G(t+h-r,x,y)-G(t-r,x,y)\right)^2dy\,dr\leq C (1+\beta^2)h^2(t-s)^{-\frac{3}{2}},
	\end{equation}
	where $C=\frac{\Gamma(3/2)}{2\pi}\max\left(\frac{\sqrt{a_2}}{a_1}, \frac{\sqrt{a_1}}{a_2}\right)$.
\end{lemma}

This estimate is consistent with the corresponding bound for the classical
stochastic heat equation obtained by Posp\'i\v{s}il and Tribe~\cite{MR2321899}.
\subsection{Existence, uniqueness and moment bounds}
The following proposition deals with the uniqueness and existence of the mild solution to equation (\ref{eq:main equation}), together with its uniform moment bounds.
\begin{proposition}\label{pro:exist}
Let
\((\Omega,\mathcal F,\{\mathcal F_t\}_{t\geq0},\mathbb P)\)
be a filtered probability space satisfying the usual conditions.
Assume that \(\sigma\) is globally Lipschitz. Then equation
\eqref{eq:main equation} admits a unique mild solution
$\{u(t,x):(t,x)\in[0,T]\times\mathbb R\}$, up to modification, satisfying, for every \(p\geq2\),
\begin{equation}\label{eq:exist}
	\sup_{(t,x)\in[0,T]\times\mathbb R}
	\mathbb E|u(t,x)|^p<\infty.
\end{equation}
\end{proposition}
The proof of the existence can be done by using the Picard iteration scheme, and for this proof, we need the following useful result. The proof can be dealt with in a similar way as in the proof of Theorem 6.4 in Dalang \emph{et al.}~\cite{MR1500166}.
\begin{lemma}[Gronwall's lemma]\label{le:gronwall}
Suppose $ \phi_1, \phi_2,\dots :[0, T]\rightarrow\mathbb{R}_+ $ are measurable and non-decreasing. Suppose also that there exists a constant $A$ such that for all integers $n\geq 1$, and all $t\in[0, T]$, we have
\begin{equation}
\phi_{n+1}(t)\leq A\int_0^t\phi_n(s)ds.
\end{equation}
Then,
\begin{equation}
\phi_n(t)\leq \phi_1(T)\frac{(At)^{n-1}}{(n-1)!}
\end{equation}
holds for $n\geq 1$ and $t\in[0, T]$.
\end{lemma}
\begin{remark}
As a consequence of Lemma \ref{le:gronwall}, any positive power of $\phi_n(t)$ is summable in $n$. Also, if $ \phi_n$ does not depend on $n$, then it follows that $\phi_n\equiv0$.
\end{remark}

\begin{proof}[Proof of Proposition \ref{pro:exist}]
The proof will be done in two steps.\\ 
We will repeatedly use the following elementary observation. If
\(F:[0,T]\to\mathbb R_+\) is non-decreasing, then, for every
\(0\leq r\leq t\leq T\),
\begin{align*}
	\int_0^r\frac{F(s)}{(r-s)^{1/2}}\,ds
	&=
	\int_0^r\frac{F(r-v)}{v^{1/2}}\,dv\\
	&\leq
	\int_0^t\frac{F(t-v)}{v^{1/2}}\,dv\\
	&=
	\int_0^t\frac{F(s)}{(t-s)^{1/2}}\,ds.
\end{align*}
\textit{Step \I: Uniqueness.}\qquad Suppose $u(t,x)$ and $v(t,x)$ both satisfy the definition of mild solution with the condition (\ref{eq:exist}). We wish to prove that $ u(t,x)$ and $v(t,x)$ are modifications of one another. We set $d(t,x):=u(t,x)-v(t,x)$, i.e., 
\[d(t,x)=\int_0^t\int_{\mathbb{R}}(\sigma(u(s,y))-\sigma(v(s,y)))G(t-s,x,y)\,W(ds,dy).\]
By the It\^o isometry and the Lipschitz continuity of \(\sigma\),
\begin{align*}
	\mathbb E|d(t,x)|^2
	&=
	\int_0^t\int_{\mathbb R}
	\mathbb E\left|
	\sigma(u(s,y))-\sigma(v(s,y))
	\right|^2
	G(t-s,x,y)^2\,dy\,ds\\
	&\leq
	M^2
	\int_0^t\int_{\mathbb R}
	\mathbb E|d(s,y)|^2
	G(t-s,x,y)^2\,dy\,ds.
\end{align*}
Define
\[
H(t)
:=
\sup_{\substack{r\in[0,t]\\x\in\mathbb R}}
\mathbb E|d(r,x)|^2.
\]
Then \(H\) is non-negative and non-decreasing. For every
\(r\in[0,t]\), Lemma~\ref{le:quasi-semigroup} yields
\begin{align*}
	\sup_{x\in\mathbb R}\mathbb E|d(r,x)|^2
	&\leq
	C
	\int_0^r
	\frac{H(s)}{(r-s)^{1/2}}\,ds.
\end{align*}
Taking the supremum over \(r\in[0,t]\) and using the preceding
elementary observation, we obtain
\[
H(t)
\leq
C
\int_0^t
\frac{H(s)}{(t-s)^{1/2}}\,ds.
\]
Choose \(q_1\in(1,2)\), and let \(q_2>2\) be its conjugate exponent.
By H\"older's inequality,
\begin{align*}
	H(t)
	&\leq
	C
	\left(
	\int_0^t(t-s)^{-q_1/2}\,ds
	\right)^{1/q_1}
	\left(
	\int_0^tH(s)^{q_2}\,ds
	\right)^{1/q_2}\\
	&\leq
	C_T
	\left(
	\int_0^tH(s)^{q_2}\,ds
	\right)^{1/q_2},
\end{align*}
where the first integral is finite because \(q_1<2\).
Consequently,
\[
H(t)^{q_2}
\leq
C
\int_0^tH(s)^{q_2}\,ds.
\]
The Gronwall lemma implies that
\[
H(t)=0,
\qquad t\in[0,T].
\]
Hence, \(u\) and \(v\) are modifications of one another.

\textit{Step \II: Existence.}\qquad 
Fix \(p\geq2\). Since \(\sigma\) is globally Lipschitz, it satisfies the
linear growth condition
\[
|\sigma(z)|
\leq
|\sigma(0)|+M|z|,
\qquad z\in\mathbb R.
\]
 Let
 \[
 u_0(t,x):=0
 \]
 and define recursively
 \[
 u_{n+1}(t,x)
 =
 \int_0^t\int_{\mathbb R}
 \sigma(u_n(s,y))G(t-s,x,y)\,W(ds,dy),
 \qquad n\geq0.
 \]
 We first establish uniform \(p\)-th moment bounds for the Picard
 iterations. Define
 \[
 K_{n,p}(t)
 :=
 \sup_{\substack{r\in[0,t]\\x\in\mathbb R}}
 \|u_n(r,x)\|_p^2.
 \]
 By the Burkholder--Davis--Gundy inequality for Walsh stochastic integrals, Minkowski's inequality,
 the linear growth of \(\sigma\), Lemma~\ref{le:quasi-semigroup}, and
 the same argument as in Step~\I\, based on the preceding elementary
 observation, we obtain
 \begin{align*}
 	K_{n+1,p}(t)
 	&\leq
 	C_p
 	\sup_{\substack{r\in[0,t]\\x\in\mathbb R}}
 	\int_0^r\int_{\mathbb R}
 	\|\sigma(u_n(s,y))\|_p^2
 	G(r-s,x,y)^2\,dy\,ds\\
 	&\leq
 	C_{p}
 	\int_0^t
 	\frac{1+K_{n,p}(s)}{(t-s)^{1/2}}\,ds.
 \end{align*}
 Using H\"older's inequality with the conjugate exponents
 \(q_1\in(1,2)\) and \(q_2>2\), we obtain
 \[
 K_{n+1,p}(t)^{q_2}
 \leq
 C_{p,T}
 \int_0^t
 \left(1+K_{n,p}(s)^{q_2}\right)\,ds.
 \]
 For \(N\geq1\), set
 \[
 \overline K_{N,p}(t)
 :=
 \max_{0\leq n\leq N}K_{n,p}(t)^{q_2}.
 \]
 Then
 \[
 \overline K_{N,p}(t)
 \leq
 C_{p,T}
 +
 C_{p,T}
 \int_0^t\overline K_{N,p}(s)\,ds.
 \]
 By Gronwall's inequality,
 \[
 \sup_{N\geq1}\sup_{t\in[0,T]}
 \overline K_{N,p}(t)
 <\infty.
 \]
 Therefore,
 \begin{equation}\label{eq:picard-uniform-p}
 	\sup_{n\geq0}
 	\sup_{(t,x)\in[0,T]\times\mathbb R}
 	\mathbb E|u_n(t,x)|^p
 	<\infty.
 \end{equation}
 
 It remains to prove that the Picard sequence is Cauchy.
Set
\[
d_n(t,x):=u_{n+1}(t,x)-u_n(t,x),
\qquad n\geq0.
\]
Then, for every \(n\geq 1\),
\[
d_n(t,x)
=
\int_0^t\int_{\mathbb R}
\left(
\sigma(u_n(s,y))-\sigma(u_{n-1}(s,y))
\right)
G(t-s,x,y)\,W(ds,dy).
\]
For every \(n\geq0\), define
\[
D_{n,p}(t)
:=
\sup_{\substack{r\in[0,t]\\x\in\mathbb R}}
\|d_n(r,x)\|_p^2.
\]
By the Burkholder--Davis--Gundy inequality for Walsh stochastic integrals, the Lipschitz continuity
of \(\sigma\), Lemma~\ref{le:quasi-semigroup}, and the same argument
as in Step~\I, we have, for every \(n\geq0\)
\[
D_{n+1,p}(t)
\leq
C_p
\int_0^t
\frac{D_{n,p}(s)}{(t-s)^{1/2}}\,ds.
\]
Applying H\"older's inequality gives
\[
D_{n+1,p}(t)^{q_2}
\leq
C_{p,T}
\int_0^tD_{n,p}(s)^{q_2}\,ds.
\]
Set
\[
\phi_n(t):=D_{n,p}(t)^{q_2}.
\]
Then
\[
\phi_n(t)
\leq
C_{p,T}
\int_0^t\phi_{n-1}(s)\,ds.
\]
By Lemma~\ref{le:gronwall},
\[
\phi_n(t)
\leq
\phi_1(T)
\frac{
	\left(C_{p,T}t\right)^{n-1}
}{
	(n-1)!
}.
\]
Consequently,
\[
\sum_{n=1}^{\infty}
\sup_{(t,x)\in[0,T]\times\mathbb R}
\|d_n(t,x)\|_p
=
\sum_{n=1}^{\infty}
D_{n,p}(T)^{1/2}
=
\sum_{n=1}^{\infty}
\phi_n(T)^{1/(2q_2)}
<\infty.
\]
Thus \(\{u_n\}_{n\geq0}\) is a Cauchy sequence in \(L^p(\Omega)\),
uniformly on \([0,T]\times\mathbb R\). Let \(u\) denote its limit.
Since \(p\geq2\), for every \(R>0\),
\[
\mathbb E\int_0^T\int_{-R}^R
|u_n(t,x)-u(t,x)|^2\,dx\,dt
\leq
2RT
\sup_{(t,x)\in[0,T]\times\mathbb R}
\mathbb E|u_n(t,x)-u(t,x)|^2
\longrightarrow0.
\]
Since each \(u_n\) is predictable and the predictable subspace is
closed in the corresponding local \(L^2\)-space, a standard
localization argument yields a predictable version of \(u\) on
\([0,T]\times\mathbb R\), which we use throughout.

Moreover, the Lipschitz continuity of \(\sigma\), the
Burkholder--Davis--Gundy inequality for Walsh stochastic integrals, and the convergence of \(u_n\) show
that
\[
\int_0^t\int_{\mathbb R}
\sigma(u_n(s,y))G(t-s,x,y)\,W(ds,dy)
\longrightarrow
\int_0^t\int_{\mathbb R}
\sigma(u(s,y))G(t-s,x,y)\,W(ds,dy)
\]
in \(L^p(\Omega)\). Therefore, \(u\) satisfies the integral equation \eqref{eq:solution}.

Finally, by \eqref{eq:picard-uniform-p} and Fatou's lemma,
\[
\sup_{(t,x)\in[0,T]\times\mathbb R}
\mathbb E|u(t,x)|^p
<\infty.
\]
Since \(p\geq2\) was arbitrary, \eqref{eq:exist} holds for every
\(p\geq2\).
\end{proof}

Zili and Zougar~\cite{MR3988829} proved the existence of the solution $\{u(t,x), t\geq0, x\in\mathbb{R}\}$ of the equation (\ref{eq:main equation}) when $\sigma\equiv1$. Also, they proved that the corresponding mild solution has a quasi-helix property, from which they deduced that its sample paths $t\longmapsto u(t,x)$ are H\"older continuous, but not differentiable.  In the present work, the coefficient \(\sigma\) is allowed to be a globally
Lipschitz function of the solution. Therefore, the increment estimates below
are derived directly for the nonlinear mild solution.	
\subsection{Increment estimates}
This subsection is devoted to moment estimates for the temporal and spatial
increments of the mild solution. The upper bounds follow from the uniform moment
estimate, the Burkholder--Davis--Gundy inequality for Walsh stochastic
integrals, and the kernel estimates in the previous subsection. Under an
additional non-degeneracy condition on \(\sigma\), we also obtain lower bounds,
which show that the nonlinear solution has a quasi-helix type local behavior.
\begin{lemma}\label{le:helix-t}
		Let \(p\geq 2\). For every \(x\in\mathbb R\), the solution
		given by \eqref{eq:solution} satisfies
		\begin{equation}\label{eq:temporal-regularity-p}
			\mathbb E|u(t,x)-u(s,x)|^p
			\leq C_{p,T}|t-s|^{\frac{p}{4}},
		\end{equation}
		for all \(s,t\in[0,T]\), where \(C_{p,T}>0\) is a constant
		independent of \(s,t\), and \(x\).
	\end{lemma}
	
	\begin{proof}
		Fix \(x\in\mathbb R\). Without loss of generality, assume that
		\(0\leq s<t\leq T\). By \eqref{eq:solution},
		we have
		\begin{align*}
			u(t,x)-u(s,x)
			={}&
			\int_0^s\int_{\mathbb R}
			\sigma(u(r,z))
			\bigl(G(t-r,x,z)-G(s-r,x,z)\bigr)
			W(dr,dz)\\
			&+
			\int_s^t\int_{\mathbb R}
			\sigma(u(r,z))G(t-r,x,z)
			W(dr,dz).
		\end{align*}
		Set
		\begin{align*}
			\Lambda_1
			&:=
			\int_0^s\int_{\mathbb R}
			\sigma(u(r,z))
			\bigl(G(t-r,x,z)-G(s-r,x,z)\bigr)
			W(dr,dz),\\
			\Lambda_2
			&:=
			\int_s^t\int_{\mathbb R}
			\sigma(u(r,z))G(t-r,x,z)
			W(dr,dz).
		\end{align*}
		Then
		\[
		u(t,x)-u(s,x)=\Lambda_1+\Lambda_2.
		\]
		
		By the Burkholder--Davis--Gundy inequality for Walsh stochastic integrals,
		together with Minkowski's  inequality, we have
		\begin{align}
			\|\Lambda_1\|_p^2
			&\leq
			C_p
			\left\|
			\int_0^s\int_{\mathbb R}
			|\sigma(u(r,z))|^2
			|G(t-r,x,z)-G(s-r,x,z)|^2
			\,dz\,dr
			\right\|_{\frac p2}
			\notag\\
			&\leq
			C_p
			\int_0^s\int_{\mathbb R}
			\|\sigma(u(r,z))\|_p^2
			|G(t-r,x,z)-G(s-r,x,z)|^2
			\,dz\,dr.
			\label{eq:I1-p-estimate}
		\end{align}
		Since the solution \(u\) is uniformly \(L^p\)-bounded on
		\([0,T]\times\mathbb R\), condition \eqref{eq:sigma} implies that
		\[
		\sup_{(r,z)\in[0,T]\times\mathbb R}
		\|\sigma(u(r,z))\|_p<\infty.
		\]
		Consequently, it remains to estimate
		\[
		\int_0^s\int_{\mathbb R}
		|G(t-r,x,z)-G(s-r,x,z)|^2
		\,dz\,dr.
		\]
		Set \(h:=t-s\). If \(0<h<s\), we split the integral as
		\begin{align*}
			&\int_0^s\int_{\mathbb R}
			|G(t-r,x,z)-G(s-r,x,z)|^2
			\,dz\,dr\\
			&=
			\int_0^{s-h}\int_{\mathbb R}
			|G(s+h-r,x,z)-G(s-r,x,z)|^2
			\,dz\,dr\\
			&\quad+
			\int_{s-h}^{s}\int_{\mathbb R}
			|G(t-r,x,z)-G(s-r,x,z)|^2
			\,dz\,dr\\
			&=:K_1+K_2.
		\end{align*}
		Applying Lemma~\ref{le:G-G2} with its parameters
		\((s,t,h)\) replaced by \((s-h,s,h)\), we obtain
		\[
		K_1
		\leq
		Ch^2\bigl(s-(s-h)\bigr)^{-3/2}
		\leq
		Ch^{1/2}.
		\]
		Moreover, by $|a-b|^2\leq2a^2+2b^2, $ with $a,b\in\mathbb{R}$ and
		Lemma~\ref{le:quasi-semigroup},
		\begin{align*}
			K_2
			&\leq
			C
			\int_{s-h}^{s}
			\left(
			(t-r)^{-1/2}+(s-r)^{-1/2}
			\right)\,dr\\
			&\leq
			Ch^{1/2}.
		\end{align*}
		If \(h\geq s\), the same bound follows directly from
		\(|a-b|^2\leq2a^2+2b^2\) and
		Lemma~\ref{le:quasi-semigroup}. Hence,
		\[
		\int_0^s\int_{\mathbb R}
		|G(t-r,x,z)-G(s-r,x,z)|^2
		\,dz\,dr
		\leq
		C|t-s|^{1/2}.
		\]
		Therefore,
		\[
		\|\Lambda_1\|_p^2
		\leq
		C_{p,T}|t-s|^{1/2},
		\]
		and consequently,
	\begin{equation}\label{eq:I1-p-final}
			\|\Lambda_1\|_p
			\leq C_{p,T}|t-s|^{\frac 14}.
		\end{equation}
		
		Similarly, applying the Burkholder--Davis--Gundy inequality and
		Minkowski's integral inequality once again, we obtain
		\begin{align}
			\|\Lambda_2\|_p^2
			&\leq
			C_p
			\int_s^t\int_{\mathbb R}
			\|\sigma(u(r,z))\|_p^2
			G(t-r,x,z)^2
			\,dz\,dr
			\notag\\
			&\leq
			C_{p,T}
			\int_s^t\int_{\mathbb R}
			G(t-r,x,z)^2
			\,dz\,dr.
			\label{eq:I2-p-estimate}
		\end{align}
		By Lemma~\ref{le:quasi-semigroup},
		\begin{align*}
			\int_s^t\int_{\mathbb R}
			G(t-r,x,z)^2
			\,dz\,dr
			&\leq
			C
			\int_s^t(t-r)^{-\frac12}\,dr\\
			&\leq
			C|t-s|^{\frac12}.
		\end{align*}
		Therefore,
		\begin{equation}\label{eq:I2-p-final}
			\|\Lambda_2\|_p
			\leq C_{p,T}|t-s|^{\frac14}.
		\end{equation}
		
		Combining \eqref{eq:I1-p-final} and \eqref{eq:I2-p-final} with the
		triangle inequality in \(L^p(\Omega)\), we obtain
		\[
		\|u(t,x)-u(s,x)\|_p
		\leq
		\|\Lambda_1\|_p+\|\Lambda_2\|_p
		\leq
		C_{p,T}|t-s|^{\frac14}.
		\]
		Raising both sides to the \(p\)-th power yields
		\[
		\mathbb E|u(t,x)-u(s,x)|^p
		\leq C_{p,T}|t-s|^{\frac p4},
		\]
		which proves \eqref{eq:temporal-regularity-p}.
\end{proof}

The following lemma provides a moment estimate for the spatial increments
of the solution.

\begin{lemma}\label{le:helix-x}
	Let \(p\geq 2\). For every \(t\in[0,T]\), the solution given by
	\eqref{eq:solution} satisfies
	\begin{equation}\label{eq:spatial-regularity}
		\mathbb E|u(t,x)-u(t,y)|^p
		\leq C_{p,T}|x-y|^{\frac p2},
	\end{equation}
	for all \(x,y\in\mathbb R\), where \(C_{p,T}>0\) is a constant
	independent of \(t,x\), and \(y\).
\end{lemma}

\begin{proof}
	By \eqref{eq:solution}, we have
	\[
	u(t,x)-u(t,y)
	=
	\int_0^t\int_{\mathbb R}
	\sigma(u(r,z))
	\bigl(G(t-r,x,z)-G(t-r,y,z)\bigr)
	W(dr,dz).
	\]
	Therefore, by the Burkholder--Davis--Gundy inequality and
	Minkowski's inequality,
	\begin{align}
		\|u(t,x)-u(t,y)\|_p^2
		&\leq
		C_p
		\left\|
		\int_0^t\int_{\mathbb R}
		|\sigma(u(r,z))|^2
		|G(t-r,x,z)-G(t-r,y,z)|^2
		\,dz\,dr
		\right\|_{\frac p2}
		\notag\\
		&\leq
		C_p
		\int_0^t\int_{\mathbb R}
		\|\sigma(u(r,z))\|_p^2
		|G(t-r,x,z)-G(t-r,y,z)|^2
		\,dz\,dr.
		\label{eq:spatial-increment-BDG}
	\end{align}
	By condition \eqref{eq:sigma} and the uniform moment estimate
	\eqref{eq:exist},
	\[
	\sup_{(r,z)\in[0,T]\times\mathbb R}
	\|\sigma(u(r,z))\|_p<\infty.
	\]
	Hence, \eqref{eq:spatial-increment-BDG} yields
	\begin{equation}\label{eq:spatial-increment-kernel}
		\|u(t,x)-u(t,y)\|_p^2
		\leq
		C_{p,T}
		\int_0^t\int_{\mathbb R}
		|G(t-r,x,z)-G(t-r,y,z)|^2
		\,dz\,dr.
	\end{equation}
	
	By the argument used to verify Condition \(H_2([0,1])\) in the
	proof of Theorem~6 of Zili and Zougar~\cite{MR4028081}, together
	with Lemma~4 therein, we have
	\begin{equation}\label{eq:kernel-spatial-increment}
		\int_0^t\int_{\mathbb R}
		|G(t-r,x,z)-G(t-r,y,z)|^2
		\,dz\,dr
		\leq C|x-y|,
		\qquad x,y\in\mathbb R.
	\end{equation}
	Although Condition \(H_2\) is stated therein on the interval
	\([0,1]\) in \cite{MR4028081}, its proof does not use the restriction
	\(x,y\in[0,1]\), and Lemma~4 therein is valid for all
	\(x,y\in\mathbb R\).

	Combining \eqref{eq:spatial-increment-kernel} and
	\eqref{eq:kernel-spatial-increment}, we obtain
	\[
	\|u(t,x)-u(t,y)\|_p^2
	\leq C_{p,T}|x-y|.
	\]
	Consequently,
	\[
	\|u(t,x)-u(t,y)\|_p
	\leq C_{p,T}|x-y|^{\frac12}.
	\]
	Raising both sides to the \(p\)-th power proves
	\eqref{eq:spatial-regularity}.
\end{proof}

\begin{assumption}
	\label{ass:nondegenerate}
	There exists a constant \(\sigma_0>0\) such that
	\[
	|\sigma(x)|\geq \sigma_0,\qquad x\in\mathbb R .
	\]
\end{assumption}
\begin{remark}
	\label{rem:why-nondegenerate}
	Assumption \ref{ass:nondegenerate} is only needed when a lower bound for the
	increments of the nonlinear solution is required. Such a lower bound cannot
	follow from the Lipschitz continuity of \(\sigma\) alone. Indeed, if
	\(\sigma(0)=0\) and the initial condition is zero, then the identically zero
	random field is a mild solution. By uniqueness, it is the only mild solution,
	and therefore no positive lower bound of the form
	\[
	\mathbb E|u(t,x)-u(s,x)|^2\geq c |t-s|^{1/2}
	\]
	can hold in general. Thus, a non-degeneracy condition is necessary if one wants
	to formulate a quasi-helix type lower estimate for the nonlinear equation.
\end{remark}
\begin{lemma}
	\label{lem:time-lower}
	Assume that \(\sigma\) is globally Lipschitz and that Assumption
	\ref{ass:nondegenerate} holds. Fix \(0<T_1<T_2\), \(x\in\mathbb R\), and
	\(p\geq2\). Then there exist constants \(c_{p,T_1,T_2,x}>0\) and
	\(h_0>0\) such that, for all \(s,t\in[T_1,T_2]\) with
	\(0<|t-s|\leq h_0\),
	\[
	\mathbb E|u(t,x)-u(s,x)|^p
	\geq
	c_{p,T_1,T_2,x}|t-s|^{p/4}.
	\]
\end{lemma}

\begin{proof}
	It is enough to prove the assertion for \(p=2\), since for \(p\geq2\),
	\[
	\mathbb E|X|^p\geq \bigl(\mathbb E|X|^2\bigr)^{p/2}
	\]
	for every random variable \(X\in L^p(\Omega)\).
	
	Assume without loss of generality that \(s<t\), and set \(h:=t-s\). By the
	mild solution and by the convention \(G(r,\cdot,\cdot)=0\) for \(r\leq0\),
	we have
	\[
	u(t,x)-u(s,x)
	=
	\int_0^t\int_{\mathbb R}
	\bigl(
	G(t-r,x,z)-G(s-r,x,z)
	\bigr)
	\sigma(u(r,z))\,W(dr,dz).
	\]
	Hence, by the It\^o isometry and Assumption \ref{ass:nondegenerate},
	\[
	\begin{aligned}
		\mathbb E|u(t,x)-u(s,x)|^2
		&=
		\int_0^t\int_{\mathbb R}
		\bigl(
		G(t-r,x,z)-G(s-r,x,z)
		\bigr)^2
		\mathbb E\bigl[\sigma^2(u(r,z))\bigr]\,dz\,dr  \\
		&\geq
		\sigma_0^2
		\int_0^t\int_{\mathbb R}
		\bigl(
		G(t-r,x,z)-G(s-r,x,z)
		\bigr)^2\,dz\,dr .
	\end{aligned}
	\]
	Let \(v\) denote the solution of the associated linear equation obtained by
	taking \(\sigma\equiv1\), namely
	\[
	v(t,x)
	=
	\int_0^t\int_{\mathbb R}
	G(t-r,x,z)\,W(dr,dz).
	\]
	Then
	\[
	\int_0^t\int_{\mathbb R}
	\bigl(
	G(t-r,x,z)-G(s-r,x,z)
	\bigr)^2\,dz\,dr
	=
	\mathbb E|v(t,x)-v(s,x)|^2 .
	\]
	By the temporal increment expansion in the linear case, see Zili and Zougar
	\cite[Lemma 5]{MR3992994}, uniformly for \(s\in[T_1,T_2]\),
	\[
	\mathbb E|v(s+h,x)-v(s,x)|^2
	=
	\sqrt{\frac{2\tau(x)}{\pi A(x)}}\,h^{1/2}
	+
	O(h^2),
	\qquad h\downarrow0.
	\]
	Since \(\tau(x)>0\) and \(A(x)>0\), there exists \(h_0>0\) such that, for
	all \(0<h\leq h_0\),
	\[
	\mathbb E|v(s+h,x)-v(s,x)|^2
	\geq
	\frac12
	\sqrt{\frac{2\tau(x)}{\pi A(x)}}\,h^{1/2}.
	\]
	Consequently,
	\[
	\mathbb E|u(t,x)-u(s,x)|^2
	\geq
	\frac{\sigma_0^2}{2}
	\sqrt{\frac{2\tau(x)}{\pi A(x)}}\,|t-s|^{1/2}.
	\]
	The conclusion for general \(p\geq2\) follows from the monotonicity of \(L^p\)
	norms.
\end{proof}

\begin{remark}
	\label{rem:role-of-lower-bound}
	The lower bound above is used only to clarify the quasi-helix type behavior of the nonlinear solution under the non-degeneracy condition. It is not required for the existence, uniqueness, or upper regularity estimates.
\end{remark}

\begin{lemma}
		\label{lem:space-lower}
		Assume that \(\sigma\) is globally Lipschitz and that Assumption
		\ref{ass:nondegenerate} holds. Fix \(0<T_1<T_2\), \(p\geq2\), and let
		\(I\subset[0,\infty)\) be a compact interval. Then there exists a constant
		\(c_{p,T_1,I}>0\) such that, for all \(t\in[T_1,T_2]\) and all
		\(x,y\in I\),
		\[
		\mathbb E|u(t,x)-u(t,y)|^p
		\geq
		c_{p,T_1,I}|x-y|^{p/2}.
		\]
	\end{lemma}
	
	\begin{proof}
		It is enough to prove the assertion for \(p=2\), since the general case
		\(p\geq2\) follows from the monotonicity of \(L^p\)-norms.
		
		By the mild formulation,
		\[
		u(t,x)-u(t,y)
		=
		\int_0^t\int_{\mathbb R}
		\bigl(
		G(t-r,x,z)-G(t-r,y,z)
		\bigr)
		\sigma(u(r,z))\,W(dr,dz).
		\]
		Hence, by the It\^o isometry for Walsh integrals and Assumption
		\ref{ass:nondegenerate},
		\[
		\begin{aligned}
			\mathbb E|u(t,x)-u(t,y)|^2
			&=
			\int_0^t\int_{\mathbb R}
			\bigl(
			G(t-r,x,z)-G(t-r,y,z)
			\bigr)^2
			\mathbb E\bigl[\sigma^2(u(r,z))\bigr]\,dz\,dr  \\
			&\geq
			\sigma_0^2
			\int_0^t\int_{\mathbb R}
			\bigl(
			G(t-r,x,z)-G(t-r,y,z)
			\bigr)^2\,dz\,dr .
		\end{aligned}
		\]
		The last integral is exactly the second moment of the spatial increment
of the corresponding linear solution. By condition \(H_1(I)\) established
in Zili and Zougar \cite{MR4028081}, applied at the fixed time \(T_1\),
there exists a constant \(c_{T_1,I}>0\) such that
\[
\int_0^{T_1}\int_{\mathbb R}
\bigl(
G(s,x,z)-G(s,y,z)
\bigr)^2\,dz\,ds
\geq
c_{T_1,I}|x-y|,
\qquad x,y\in I.
\]
Moreover, by the change of variables \(s=t-r\), for every
\(t\in[T_1,T_2]\),
\[
\begin{aligned}
&\int_0^t\int_{\mathbb R}
\bigl(
G(t-r,x,z)-G(t-r,y,z)
\bigr)^2\,dz\,dr\\
&\qquad=
\int_0^t\int_{\mathbb R}
\bigl(
G(s,x,z)-G(s,y,z)
\bigr)^2\,dz\,ds\\
&\qquad\geq
\int_0^{T_1}\int_{\mathbb R}
\bigl(
G(s,x,z)-G(s,y,z)
\bigr)^2\,dz\,ds
\geq c_{T_1,I}|x-y|.
\end{aligned}
\]
		Therefore,
		\[
		\mathbb E|u(t,x)-u(t,y)|^2
		\geq
		\sigma_0^2 c_{T_1,I}|x-y|.
		\]
		The proof is complete.
	\end{proof}
	
	\begin{corollary}
		\label{cor:two-sided-increments}
		Assume that \(\sigma\) is globally Lipschitz and that Assumption
		\ref{ass:nondegenerate} holds.
		
		\noindent
		\((i)\) Fix \(0<T_1<T_2\), \(x\in\mathbb R\), and \(p\ge2\). Then there exist
		constants \(0<c_{p,T_1,T_2,x}\le C_{p,T_1,T_2}<\infty\) and \(h_0>0\)
		such that, for all \(s,t\in[T_1,T_2]\) with \(0<|t-s|\le h_0\),
		\[
		c_{p,T_1,T_2,x}|t-s|^{p/4}
		\le
		\mathbb E|u(t,x)-u(s,x)|^p
		\le
		C_{p,T_1,T_2}|t-s|^{p/4}.
		\]
		In particular, taking \(p=2\), the temporal increments satisfy the quasi-helix
		type estimate
		\[
		\mathbb E|u(t,x)-u(s,x)|^2
		\asymp
		|t-s|^{1/2},
		\qquad 0<|t-s|\le h_0 .
		\]
		
		\noindent
		\((ii)\) Fix \(0<T_1<T_2\), \(p\ge2\), and let \(I\subset[0,\infty)\) be a
		compact interval. Then there exist constants
		\(0<c_{p,T_1,T_2,I}\le C_{p,T_1,T_2}<\infty\) such that, for all
		\(t\in[T_1,T_2]\) and all \(x,y\in I\),
		\[
		c_{p,T_1,T_2,I}|x-y|^{p/2}
		\le
		\mathbb E|u(t,x)-u(t,y)|^p
		\le
		C_{p,T_1,T_2}|x-y|^{p/2}.
		\]
		In particular, taking \(p=2\),
		\[
		\mathbb E|u(t,x)-u(t,y)|^2
		\asymp
		|x-y|,
		\qquad x,y\in I.
		\]
	\end{corollary}
	
	\begin{proof}
		The lower bounds follow from Lemma \ref{lem:time-lower} and Lemma
		\ref{lem:space-lower}, respectively. The upper bounds follow from Lemma
		\ref{le:helix-t} and Lemma \ref{le:helix-x}.
	\end{proof}

The upper increment estimates will be used in the proof of the
main result, while the lower bounds describe the quasi-helix type
behavior of the solution. As a direct consequence of the preceding temporal and spatial increment
estimates, we obtain the joint continuity of the solution.

\begin{corollary}\label{cor:joint-continuity}
	The solution \(u\) admits a jointly continuous modification on
	\([0,T]\times\mathbb R\).
\end{corollary}

\begin{proof}
	Fix \(R>0\) and let \(p>8\). By the triangle inequality,
	Lemma~\ref{le:helix-t}, and Lemma~\ref{le:helix-x}, for all
	\(s,t\in[0,T]\) and \(x,y\in[-R,R]\), we have
	\begin{align*}
		\mathbb E|u(t,x)-u(s,y)|^p
		&\leq
		C_p\mathbb E|u(t,x)-u(s,x)|^p
		+
		C_p\mathbb E|u(s,x)-u(s,y)|^p\\
		&\leq
		C_{p,T}
		\left(
		|t-s|^{p/4}+|x-y|^{p/2}
		\right).
	\end{align*}
	Let
	\[
	\rho\big((t,x),(s,y)\big)
	:=
	\sqrt{|t-s|^2+|x-y|^2}.
	\]
	Whenever \(\rho((t,x),(s,y))\leq1\), it follows that
	\[
	\mathbb E|u(t,x)-u(s,y)|^p
	\leq
	C_{p,T,R}\,
	\rho\big((t,x),(s,y)\big)^{p/4}.
	\]
	Since \(p/4>2\), the two-parameter Kolmogorov continuity theorem implies
	that \(u\) admits a jointly continuous modification on
	\([0,T]\times[-R,R]\). By a standard localization argument over \(R\in\mathbb N\), these local
	modifications can be chosen consistently. Hence \(u\) admits a jointly continuous modification on
	\([0,T]\times\mathbb R\).
\end{proof}

\section{Temporal Variation}
In this section, we fix \(0<T_1<T_2\) and \(x\in\mathbb R\), and define
the temporal increment by
\[
\Delta(t,\delta):=u(t,x)-u(t-\delta,x),
\]
where
\[
\delta:=\frac{T_2-T_1}{n}.
\]
For every $n\geq1$, we set the temporal quartic variation of the path
\(t\mapsto u(t,x)\) by
\begin{equation}\label{eq:V}
V_{n,x}(u):=\sum_{i=1}^{n}(u(t_{i},x)-u(t_{i-1},x))^4,
\end{equation}
with
\begin{equation}\label{eq:t_i}
t_i=T_1+i\delta, \quad\text{for}\quad i=0,\dots, n.
\end{equation}
In this paper, we mainly analyze the asymptotic behavior as $n$ tends to infinity of the sequence $(V_{n,x}, n\geq1, x\in\mathbb{R})$. For the linear case, we refer to Zili and Zougar \cite{MR3992994}, from which we will use several useful results.

The key idea is that the main contribution to $ \Delta(t+\delta,\delta) $ comes from the noise near time $ t $. We will approximate the increment by $ \sigma(u(t(\delta),x))\widetilde{\Delta}(t,\delta)$ where $t(\delta)$ is a point close to $t$ given by $ t(\delta)=t-\delta^{4/5} $. Throughout this section, we use the convention that
$G(t,x,z)=0$ when $t\leq0$. Let us set the "modified temporal increment" of $u(t,x)$ given by
\begin{multline*}
\widetilde{\Delta}(t,\delta)=\int_0^{t(\delta)}\int_{\mathbb{R}}(G(t+\delta-r,x,z)-G(t-r,x,z))\widetilde{W}(dz,dr)\\
+\int_{t(\delta)}^{t+\delta}\int_{\mathbb{R}}(G(t+\delta-r,x,z)-G(t-r,x,z))W(dz,dr)
\end{multline*}
where \(\widetilde W\) is a space--time white noise independent of \(W\). Notice that \(\widetilde{\Delta}(t,\delta)\) is a centered Gaussian random
variable with the same distribution as the corresponding linear Gaussian
increment. Moreover, it is independent of \(\mathcal F_{t(\delta)}\). The value of $ t(\delta) $ is chosen to optimize the estimation in the later theorem, which shows the approximation is valid. 

The next lemma gives an approximation about the temporal increment of the solution (\ref{eq:solution}).
\begin{lemma}\label{le:imp esti}
	For any $ T>0 $, there exists a constant \(C_{T}>0\) such that
	\[
	\mathbb{E}[|\Delta(t+\delta,\delta)-\sigma(u(t(\delta),x))
	\widetilde{\Delta}(t,\delta)|^2]
	\leq C_{T}\delta^{4/5},
	\]
	for all \(x\in\mathbb{R}\) and \(t,\delta>0\) satisfying
	\[
	t+\delta\leq T
	\qquad\text{and}\qquad
	0<\delta^{4/5}\leq t\wedge1.
	\]
\end{lemma}
\begin{proof}
	Note again that $ \widetilde{\Delta}(t,\delta) $ is independent of the sigma-field generated by \(\mathcal F_{t(\delta)}\), so we write the difference $\Delta(t+\delta,\delta)-\sigma(u(t(\delta),x))\widetilde{\Delta}(t,\delta)$ as
	\begin{align*}
		&\qquad\Delta(t+\delta,\delta)-\sigma(u(t(\delta),x))\widetilde{\Delta}(t,\delta)\\
		&=\int_0^{t+\delta} \int_{\mathbb{R}} G(t+\delta-r,x,z) \sigma(u(r,z)) W(dz, dr)-\int_0^t \int_{\mathbb{R}} G(t-r,x,z) \sigma(u(r,z)) W(dz,dr) \\
		&\quad-\sigma(u(t(\delta), x))\int_0^{t(\delta)} \int_{\mathbb{R}}\left(G(t+\delta-r,x,z)-G(t-r,x,z)\right) \widetilde{W}(dz, dr) \\
		&\quad-\sigma(u(t(\delta), x)) \int_{t(\delta)}^{t+\delta} \int_{\mathbb{R}}\left(G(t+\delta-r,x,z)-G(t-r,x,z)\right) W(dz, dr) \\
		&=\int_0^{t(\delta)} \int_{\mathbb{R}}\left(G(t+\delta-r,x,z)-G(t-r,x,z)\right) \sigma(u(r,z))W(dz, dr) \\
		&\quad+\int_{t(\delta)}^{t+\delta} \int_{\mathbb{R}}\left(G(t+\delta-r,x,z)-G(t-r,x,z)\right) (\sigma(u(r,z))-\sigma(u(t(\delta),x))) W(dz, dr) \\
		&\quad-\int_{0}^{t(\delta)} \int_{\mathbb{R}}\left(G(t+\delta-r,x,z)-G(t-r,x,z)\right) \sigma(u(t(\delta), x)) \widetilde{W}(dz, dr) \\
		&:=I_1+I_2-I_3.
	\end{align*}
	
	Consequently, we have
	\begin{equation*}
		\mathbb{E}[|\Delta(t+\delta,\delta)-\sigma(u(t(\delta),x))
		\widetilde{\Delta}(t,\delta)|^2]
		=\mathbb{E}[|I_1+I_2-I_3|^2]
		\leq C(\mathbb{E}|I_1|^2+\mathbb{E}|I_2|^2+\mathbb{E}|I_3|^2) .
	\end{equation*}
	
	Now we estimate each of these terms respectively. By using the It\^o isometry, (\ref{eq:exist}) and Lemma \ref{le:G-G2}, we have
	\begin{align}
		\mathbb{E}|I_1|^2
		&=
		\int_0^{t(\delta)}\int_{\mathbb{R}}
		\left(
		G(t+\delta-r,x,z)-G(t-r,x,z)
		\right)^2
		\mathbb{E}\left|\sigma(u(r,z))\right|^2
		\,dz\,dr \notag\\
		&\leq
		C_{T}
		\int_0^{t(\delta)}\int_{\mathbb{R}}
		\left(
		G(t+\delta-r,x,z)-G(t-r,x,z)
		\right)^2
		\,dz\,dr \notag\\
		&\leq
		C_{T}
		\delta^2
		\left(t-t(\delta)\right)^{-3/2}\notag\\
		&=
		C_{T}\delta^{4/5}.
		\label{eq:est.I_1}
	\end{align}
	
	Moreover, since
	\(\sigma(u(t(\delta),x))\) is independent of the stochastic integral
	with respect to \(\widetilde W\), the It\^o isometry and
	Lemma~\ref{le:G-G2} yield
	\begin{align}
		\mathbb E|I_3|^2
		&=
		\mathbb E\left|\sigma(u(t(\delta),x))\right|^2
		\int_0^{t(\delta)}\int_{\mathbb R}
		\left(
		G(t+\delta-r,x,z)-G(t-r,x,z)
		\right)^2
		\,dz\,dr \notag\\
		&\leq
		C_{T}\delta^{4/5}.
		\label{eq:est.I_3}
	\end{align}

	The term $\mathbb{E}|I_2|^2$ can be written as
	\[
	\int_{t(\delta)}^{t+\delta}\int_{\mathbb{R}}
	(G(t+\delta-r,x,z)-G(t-r,x,z))^2
	\mathbb E\left|
	\sigma(u(r,z))-\sigma(u(t(\delta),x))
	\right|^2\,dz\,dr.
	\]
	By using the property in Lemma \ref{le:helix-t} and Lemma \ref{le:helix-x}, we have
	\begin{eqnarray*}
		&&\mathbb{E}|I_2|^2\leq C_T\int_{t(\delta)}^{t+\delta}\int_{\mathbb{R}}(G(t+\delta-r,x,z)-G(t-r,x,z))^2(\delta^{2/5}+|x-z|)dz\,dr\\
		&\leq& C_T\int_{t(\delta)}^{t+\delta}\int_{\mathbb{R}}(G(t+\delta-r,x,z)^2+G(t-r,x,z)^2)(\delta^{2/5}+|x-z|)dz\,dr\\
		&:=&O_1+O_2,
	\end{eqnarray*}
	where
	\[
	O_1:=C_T\int_{t(\delta)}^{t+\delta}\int_{\mathbb{R}}G(t+\delta-r,x,z)^2(\delta^{2/5}+|x-z|)dz\,dr,
	\]
	and
	\[
	O_2:=C_T\int_{t(\delta)}^{t+\delta}\int_{\mathbb{R}}G(t-r,x,z)^2(\delta^{2/5}+|x-z|)dz\,dr.
	\]
	(In fact, the estimate of the second term $O_2$ is entirely similar to $O_1$, and we omit it.) Our primary focus will be on analyzing the first item.
	For this purpose, let us rewrite the $O_1$ as
	\[
	O_1=C_T\delta^{2/5}\int_{t(\delta)}^{t+\delta}\int_{\mathbb{R}}G(t+\delta-r,x,z)^2dz\,dr+C_T\int_{t(\delta)}^{t+\delta}\int_{\mathbb{R}}G(t+\delta-r,x,z)^2|x-z|dz\,dr:=O_{11}+O_{12}.
	\]
	By applying Lemma \ref{le:quasi-semigroup}, we have
	\begin{equation*}
		O_{11}\leq C_T\delta^{2/5}\int_{t(\delta)-\delta}^{t}\frac{1}{\sqrt{t-r}}dr\leq C_T(\delta^{9/10}+\delta^{4/5}).
	\end{equation*}
	Furthermore,
	\begin{align*}
		O_{12}&\leq C_T\int_{t(\delta)}^{t+\delta}\int_{\mathbb{R}}\left|\frac{1}{\sqrt{t+\delta-r}}\exp\left(-\frac{(f(x)-f(z))^2}{2(t+\delta-r)}\right)\right|^2|x-z|dz\,dr\\
		&\leq C_T\int_{t(\delta)-\delta}^{t}\int_{\mathbb{R}}\left|\frac{1}{\sqrt{t-r}}\exp\left(-\frac{(f(x)-f(z))^2}{2(t-r)}\right)\right|dz\,dr\\
		&\leq C_T(t+\delta-t(\delta))\leq C_T(\delta+\delta^{4/5}),
	\end{align*}
	where the first inequality follows from Lemma \ref{le:Gleq}, and the second one follows from the fact that
	\begin{multline*}
		\sup\limits_{r, z}\left|\frac{1}{\sqrt{t-r}}\exp\left(-\frac{(f(x)-f(z))^2}{2(t-r)}\right)\right||x-z|\\
		\leq \sup\limits_{r, z}\left|\frac{1}{\sqrt{t-r}}\exp\left(-\frac{(x-z)^2}{2(a_1\vee a_2)(t-r)}\right)\right||x-z|<\infty.
	\end{multline*}
	Further, the last estimate follows from the elementary change of variables on the two
	half-lines. Indeed, 
	\begin{multline*}
		\int_{\mathbb{R}}\exp\left(-\frac{(f(x)-f(z))^2}{2(t-r)}\right)dz=\int_{-\infty}^0\exp\left(-\frac{(f(x)-\frac{z}{\sqrt{a_1}})^2}{2(t-r)}\right)dz\\
		+\int_0^{\infty}\exp\left(-\frac{(f(x)-\frac{z}{\sqrt{a_2}})^2}{2(t-r)}\right)dz\leq (\sqrt{a_1}\vee \sqrt{a_2})\sqrt{2\pi(t-r)}
	\end{multline*}
	(Actually, a similar argument has been made in the proof of Corollary 2.1 in Mishura \emph{et al.}~\cite{MR4192903}.) 
	Therefore, we have
	\begin{equation}\label{eq:est.I_2}
		\mathbb{E}|I_2|^2\leq  C_T\,\delta^{4/5},
	\end{equation}
	since $\delta$ is small enough.
	
	Combined with (\ref{eq:est.I_1}), (\ref{eq:est.I_3}) and (\ref{eq:est.I_2}), we obtain the desired result.
\end{proof}
Using this approximation, we now prove the temporal quartic variation for the
nonlinear equation. The Gaussian moment identities will be applied only to the
auxiliary Gaussian increment, while the nonlinear increment itself is handled
through Lemma \ref{le:imp esti}.
\begin{remark}
	\label{rem:role-of-frozen-approximation}
	Lemma \ref{le:imp esti} is the key step that allows us to pass from the
	nonlinear increment to a Gaussian approximation. We emphasize that the
	nonlinear increment
	\[
	\Delta(t+\delta,\delta)
	=
	u(t+\delta,x)-u(t,x)
	\]
	is not a Gaussian random variable in general. Instead, it is approximated in
	\(L^2(\Omega)\) by the frozen-coefficient Gaussian proxy
	\[
	\sigma(u(t(\delta),x))\widetilde{\Delta}(t,\delta).
	\]
	Here \(\sigma(u(t(\delta),x))\) is measurable with respect to
	\(\mathcal F_{t(\delta)}\), while the auxiliary increment
	\(\widetilde{\Delta}(t,\delta)\) is independent of \(\mathcal F_{t(\delta)}\)
	and has the same distribution as the corresponding linear Gaussian increment.
	
	Therefore, in the proof of Theorem \ref{th:tem vari}, Gaussian moment
	identities and the linear variance asymptotics are applied only to the
	auxiliary Gaussian increment \(\widetilde{\Delta}(t,\delta)\), and not to the
	nonlinear increment itself. The estimate in Lemma \ref{le:imp esti} ensures
	that the replacement error is negligible in the quartic variation. Consequently,
	the limiting constant
	\[
	\frac{6\tau(x)}{\pi A(x)}
	\]
	comes from the fourth moment of the auxiliary Gaussian increment, whereas the
	nonlinearity is retained in the limiting integral through the factor
	\(\sigma^4(u(t,x))\).
\end{remark}
\begin{theorem}\label{th:tem vari}
	Fix $0<T_1<T_2$ and $x\in\mathbb{R}$. Define for $i=0, 1,\dots, n$ a time grid by $t_i=T_1+i\delta$, where $\delta=\frac{T_2-T_1}{n}$. Then the following limit holds in probability
	\begin{equation}\label{eq:tem-vari}
		\lim\limits_{n\rightarrow\infty}\sum_{i=1}^{n}\left(u(t_i, x)-u(t_{i-1}, x)\right)^4=\frac{6\tau(x)}{\pi A(x)}\int_{T_1}^{T_2}\sigma^4(u(r,x))dr\,,
	\end{equation}
	with $\tau(x)=\eta^2\mathbbm{1}_{\{x=0\}}+\mathbbm{1}_{\{x\neq0\}}$ and $\eta=\frac{1}{2}\left((1-\beta)^2+\sqrt{\frac{a_1}{a_2}}(1+\beta)^2\right)$.
	Moreover, we have
	\begin{equation}\label{eq:L1con}
		\mathbb{E}\left|\sum_{i=1}^{n}\left(u(t_i, x)-u(t_{i-1}, x)\right)^4-\frac{6\tau(x)}{\pi A(x)}\int_{T_1}^{T_2}\sigma^4(u(r,x))dr\right|\leq C_{T_1,T_2,x}\,n^{-\frac{3}{20}},
	\end{equation}
	for all sufficiently large \(n\)
\end{theorem}
\begin{proof}
	To prove \eqref{eq:tem-vari}, it suffices to show that
	\[
	\sum_{i=1}^{n}(u(t_i, x)-u(t_{i-1},x))^4-\frac{6\tau(x)}{\pi A(x)}\int_{T_1}^{T_2}\sigma^4(u(r,x))dr\rightarrow0\,\,\, \text{as}\,\,\, n\rightarrow\infty.
	\]
	We use one further modification of the approximation introduced above.
	Let \(\{\widetilde W_i\}_{i\geq1}\) be a sequence of independent
	space--time white noises, independent of \(W\). For each \(i=1,\ldots,n\),
	let \(\widetilde\Delta_i(t_{i-1},\delta)\) be defined in the same way as
	\(\widetilde\Delta(t_{i-1},\delta)\), with \(\widetilde W\) replaced by
	\(\widetilde W_i\). Then, for each \(i\),
	\(\widetilde\Delta_i(t_{i-1},\delta)\) has the same distribution as the
	corresponding linear Gaussian increment and is independent of
	\(\mathcal F_{t_{i-1}(\delta)}\). Moreover, the estimate in
	Lemma~\ref{le:imp esti} remains valid with
	\(\widetilde\Delta(t_{i-1},\delta)\) replaced by
	\(\widetilde\Delta_i(t_{i-1},\delta)\).
	
	We need to break the required convergence into three parts as follows
	\begin{align}
		&{}\sum_{i=1}^{n}\Delta^4(t_i,\delta)-\frac{6\tau(x)}{\pi A(x)}\int_{T_1}^{T_2}\sigma^4(u(r,x))dr\notag\\
		&=\sum_{i=1}^{n}\left[\Delta^4(t_i,\delta)-\sigma^4(u(t_{i-1}(\delta),x))\widetilde{\Delta}_i^4(t_{i-1},\delta)\right]\label{eq:part11}\\
		&\quad+\sum_{i=1}^{n}\sigma^4(u(t_{i-1}(\delta),x))\left(\widetilde{\Delta}_i^4(t_{i-1},\delta)-\frac{6\delta}{\pi A(x)}\tau(x)\right)\label{eq:part22}\\
		&\quad+\frac{6\delta}{\pi A(x)}\tau(x)\sum_{i=1}^{n}\sigma^4(u(t_{i-1}(\delta),x))-\frac{6\tau(x)}{\pi A(x)}\int_{T_1}^{T_2}\sigma^4(u(r,x))dr\label{eq:part33}\\
		&=:J_1+J_2+J_3,
	\end{align}
	for all $x\in\mathbb{R}, \delta>0, n\geq 1$ and $0<T_1<T_2$, where $t_{i-1}(\delta)=t_{i-1}-\delta^{4/5}$. 
		We first consider the third term in the decomposition. For simplicity, set
		\[
		c_x:=\frac{6\tau(x)}{\pi A(x)}
		\]
		and write
		\[
		J_3
		:=
		c_x\delta\sum_{i=1}^{n}\sigma^4(u(t_{i-1}(\delta),x))
		-c_x\int_{T_1}^{T_2}\sigma^4(u(r,x))\,dr.
		\]
		Since \(t_i-t_{i-1}=\delta\), we have
		\[
		J_3
		=
		c_x\sum_{i=1}^{n}
		\int_{t_{i-1}}^{t_i}
		\left[
		\sigma^4(u(t_{i-1}(\delta),x))
		-\sigma^4(u(r,x))
		\right]\,dr.
		\]
		
		We first prove the almost sure convergence of \(J_3\). Since \(T_1>0\)
		and \(t_{i-1}(\delta)=t_{i-1}-\delta^{4/5}\), for all sufficiently large \(n\),
		\[
		t_{i-1}(\delta)\in\left[\frac{T_1}{2},T_2\right],
		\qquad i=1,\ldots,n.
		\]
		By the sample path continuity of \(u\) and the continuity of \(\sigma\),
		for almost every \(\omega\), the map
		\[
		s\longmapsto \sigma^4(u(s,x;\omega))
		\]
		is uniformly continuous on the compact interval \([T_1/2,T_2]\).
		Moreover, for every \(r\in[t_{i-1},t_i]\),
		\[
		|t_{i-1}(\delta)-r|
		\le |t_{i-1}-r|+\delta^{4/5}
		\le \delta+\delta^{4/5}.
		\]
		Consequently,
		\[
		|J_3|
		\le
		c_x(T_2-T_1)
		\sup_{\substack{s,r\in[T_1/2,T_2]\\
				|s-r|\le \delta+\delta^{4/5}}}
		\left|
		\sigma^4(u(s,x))
		-\sigma^4(u(r,x))
		\right|
		\longrightarrow 0
		\]
		almost surely as \(n\to\infty\). Hence,
		\[
		J_3\xrightarrow{\mathrm{a.s.}}0.
		\]
		
		We next establish an \(L^1(\Omega)\)-estimate. For $a,b\in\mathbb{R}$, by the fact
		\[
		|a^4-b^4|
		\le
		|a-b|
		\left(
		|a|^3+|a|^2|b|+|a||b|^2+|b|^3
		\right),
		\]
		the Cauchy--Schwarz inequality, the Lipschitz continuity of \(\sigma\),
		and the uniform moment bounds of the solution, we obtain
		\[
		\begin{aligned}
			&\mathbb E\left|
			\sigma^4(u(t_{i-1}(\delta),x))
			-\sigma^4(u(r,x))
			\right|\\
			&\qquad\le
			C_{T_1,T_2}
			\left(
			\mathbb E
			|u(t_{i-1}(\delta),x)-u(r,x)|^2
			\right)^{1/2}.
		\end{aligned}
		\]
		It follows from the temporal increment estimate
		\[
		\mathbb E|u(s,x)-u(r,x)|^2
		\le C|s-r|^{1/2}
		\]
		that
		\[
		\mathbb E\left|
		\sigma^4(u(t_{i-1}(\delta),x))
		-\sigma^4(u(r,x))
		\right|
		\le
		C_{T_1,T_2}
		|t_{i-1}(\delta)-r|^{1/4}.
		\]
		Therefore,
		\[
		\begin{aligned}
			\mathbb E|J_3|
			&\le
			C_{T_1,T_2}
			\sum_{i=1}^{n}
			\int_{t_{i-1}}^{t_i}
			|t_{i-1}(\delta)-r|^{1/4}\,dr\\
			&\le
			C_{T_1,T_2}
			n\delta
			\left(\delta+\delta^{4/5}\right)^{1/4}.
		\end{aligned}
		\]
		Since \(\delta<1\) for all sufficiently large \(n\), we have
		\(\delta\le\delta^{4/5}\), and hence
		\begin{equation}\label{eq:J_3-L1}
			\mathbb E|J_3|
			\le
			C_{T_1,T_2}\delta^{1/5}
			\le
			C_{T_1,T_2}n^{-1/5}.
		\end{equation}
	
	Further, according to the Cauchy-Schwarz inequality and Lemma \ref{le:imp esti}, we have
	\begin{align*}
		\mathbb{E}|J_1|\leq &\sum_{i=1}^{n}\mathbb{E}\left|\Delta^4(t_i,\delta)-\sigma^4(u(t_{i-1}(\delta),x))\widetilde{\Delta}_i^4(t_{i-1},\delta)\right|\\
		&\leq
		C\sum_{i=1}^{n}
		\left(
		\mathbb E|\Delta(t_i,\delta)|^6
		+
		\mathbb E\left|
		\sigma(u(t_{i-1}(\delta),x))
		\widetilde\Delta_i(t_{i-1},\delta)
		\right|^6
		\right)^{1/2}\\
		&\qquad\times
		\left(
		\mathbb E\left|
		\Delta(t_i,\delta)
		-\sigma(u(t_{i-1}(\delta),x))
		\widetilde\Delta_i(t_{i-1},\delta)
		\right|^2
		\right)^{1/2}\\
		&\leq C_{T_1,T_2}\delta^{2/5}\sum_{i=1}^{n}\left(
		\mathbb E|\Delta(t_i,\delta)|^6
		+
		\mathbb E\left|
		\sigma(u(t_{i-1}(\delta),x))
		\widetilde\Delta_i(t_{i-1},\delta)
		\right|^6
		\right)^{1/2}.
	\end{align*} 
	In fact, for all sufficiently large \(n\), we have
	\[
	t_{i-1}+\delta=t_i\leq T_2
	\]
	and
	\[
	\delta^{4/5}\leq T_1\leq t_{i-1},
	\qquad i=1,\ldots,n.
	\]
	Therefore, Lemma~\ref{le:imp esti} is applicable with
	\(t=t_{i-1}\) and \(T=T_2\). Moreover, \(\delta^{4/5}\leq1\) for all sufficiently large \(n\).
	
	Since \(\widetilde\Delta_i(t_{i-1},\delta)\) is a centered Gaussian
		random variable having the same distribution as the corresponding
		linear Gaussian increment, we know that (see, Zili and Zougar~\cite{MR3992994})
		\[
		\mathbb E\left[
		\widetilde\Delta_i^2(t_{i-1},\delta)
		\right]=\delta^{1/2}\sqrt{\frac{2\tau(x)}{\pi A(x)}}+O(\delta^2),
		\]
	where $O(\delta^2)$ denotes a quantity that is bounded by $c\delta^2$ if $\delta$ is close enough to $0$, with the same constant $c$ for all $t\in[T_1,T_2]$. Since $\widetilde{\Delta}_i(t_{i-1}, \delta)$ has the same Gaussian distribution as the corresponding linear Gaussian increment, we get
	\[
	\mathbb{E}\left|
		\widetilde{\Delta}_i(t_{i-1},\delta)
		\right|^6=15\left(\delta^{\frac{1}{2}}\sqrt{\frac{2\tau(x)}{\pi A(x)}}+O(\delta^2)\right)^3=15\left(\frac{2\tau(x)}{\pi A(x)}\right)^{3/2}\delta^{3/2}
	+O(\delta^3).
	\] 
	Moreover, by Lemma \ref{le:helix-t}, we obtain
		\[
		\mathbb{E}[\Delta^6(t_i,\delta)]\leq C\delta^{3/2}.
		\]
	Hence, we have
	\begin{equation*}
		\sum_{i=1}^{n}\mathbb{E}\left|\Delta^4(t_i,\delta)-\sigma^4(u(t_{i-1}(\delta),x))\widetilde{\Delta}_i^4(t_{i-1},\delta)\right|
		\leq C_{T_1,T_2}\delta^{\frac{2}{5}}\,n\,\delta^{\frac{3}{4}}\longrightarrow0
	\end{equation*}
	for all $x\in\mathbb{R} $ and $0<T_1<T_2$, as $n$ tends to infinity, where the above inequality follows from the independence between $\sigma(u(t_{i-1}(\delta), x))$ and $\widetilde{\Delta}_i(t_{i-1},\delta)$. It implies that
	\begin{equation}\label{eq:J_1-L1}
		\mathbb{E}|J_1|\leq C_{T_1,T_2}n^{-\frac{3}{20}}.
	\end{equation}
	
	It remains to prove that
		\[
		\sum_{i=1}^{n}
		\sigma^4(u(t_{i-1}(\delta),x))
		\left(
		\widetilde\Delta_i^4(t_{i-1},\delta)
		-c_x\delta
		\right)
		\longrightarrow 0
		\]
		in \(L^2(\Omega)\). For simplicity, we denote by
		\[
		Z_i:=\widetilde\Delta_i^4(t_{i-1},\delta)-c_x\delta,
		\qquad
		\sigma_i:=\sigma(u(t_{i-1}(\delta),x)).
		\]
		Then it is enough to show that
		\[
		\sum_{i,j=1}^{n}\mathbb E[\sigma_i^4\sigma_j^4 Z_iZ_j]\to0.
		\]
		We split the double sum into the diagonal part, the near-diagonal part
		and the far-off-diagonal part.
		
		For the diagonal terms, using the independence between
		\(\sigma_i\) and \(\widetilde\Delta_i(t_{i-1},\delta)\), together with the moment
		bounds of the mild solution and the Gaussian moment estimates for
		\(\widetilde\Delta_i(t_{i-1},\delta)\), we obtain
		\[
		\sum_{i=1}^n
		\mathbb E[\sigma_i^8 Z_i^2]
		\le
		C_{T_1,T_2}\sum_{i=1}^n \mathbb E[Z_i^2]
		\le
		C_{T_1,T_2} n\delta^2\to0.
		\]
		Indeed, since \(\widetilde\Delta_i(t_{i-1},\delta)\) is centered Gaussian and
		\[
		\mathbb E[\widetilde\Delta_i^2(t_{i-1},\delta)]
		=
		\delta^{1/2}\sqrt{\frac{2\tau(x)}{\pi A(x)}}+O(\delta^2),
		\]
		we have
		\[
		\mathbb E[\widetilde\Delta_i^4(t_{i-1},\delta)]
		=
		c_x\delta+O(\delta^{5/2}),
		\qquad
		\mathbb E[\widetilde\Delta_i^8(t_{i-1},\delta)]
		=O(\delta^2),
		\]
		and hence \(\mathbb E[Z_i^2]\le C\delta^2\).
		
		For the near-diagonal off-diagonal terms, namely \(|i-j|\le \lfloor\delta^{-1/5}\rfloor+1\), since \(\widetilde\Delta_i(t_{i-1},\delta)\) is centered Gaussian with variance
		of order \(\delta^{1/2}\), the Gaussian moment formula yields
		\[
		\mathbb E\big|\widetilde\Delta_i(t_{i-1},\delta)\big|^{16}
		\leq C\delta^4.
		\]
		Hence, for $a,b\in\mathbb{R}$, using \(|a-b|^4\leq8(|a|^4+|b|^4)\), we obtain
		\[
		\mathbb E|Z_i|^4\leq C\delta^4.
		\]
		Therefore, by H\"older's inequality and the uniform moment bounds of the mild
		solution,
		\[
		\left|
		\mathbb E[\sigma_i^4\sigma_j^4Z_iZ_j]
		\right|
		\le
		\big(\mathbb E|\sigma_i|^{16}\big)^{1/4}
		\big(\mathbb E|\sigma_j|^{16}\big)^{1/4}
		\big(\mathbb E|Z_i|^4\big)^{1/4}
		\big(\mathbb E|Z_j|^4\big)^{1/4}\le C_{T_1,T_2}\delta^2.
		\]
		Since the number of pairs satisfying \(0<|i-j|\leq \lfloor\delta^{-1/5}\rfloor+1\) is bounded by
\(2n(\lfloor\delta^{-1/5}\rfloor+1)\), it follows that
\[
\begin{aligned}
	\sum_{\substack{1\leq i,j\leq n\\0<|i-j|\leq \lfloor\delta^{-1/5}\rfloor+1}}
	\left|
	\mathbb E[\sigma_i^4\sigma_j^4Z_iZ_j]
	\right|
	&\qquad\leq
	C_{T_1,T_2,x}\,n(\lfloor\delta^{-1/5}\rfloor+1)\delta^2\\
	&\qquad\leq
	C_{T_1,T_2,x}
	n\bigl(\delta^{-1/5}+1\bigr)\delta^2\\
	&\qquad\leq
	C_{T_1,T_2,x}n^{-4/5},
\end{aligned}
\]
where we used \(\delta=(T_2-T_1)/n\).
		
	 It remains to treat the far-off-diagonal terms. Suppose, for instance,
that
\[
	i-j\geq 	\left\lfloor\delta^{-1/5}\right\rfloor+2.
\]
Then
\[
\begin{aligned}
	t_{i-1}(\delta)-t_j
	&=
	(i-j-1)\delta-\delta^{4/5}\\
	&=
	\delta\left(i-j-1-\delta^{-1/5}\right)>0.
\end{aligned}
\]
Consequently,
\[
	t_j<t_{i-1}(\delta).
\] 
Therefore \(Z_j\), \(\sigma_j\), and \(\sigma_i\) are measurable with
		respect to
		\[
		\mathcal G_{i,j}:=\mathcal F_{t_{i-1}(\delta)}
		\vee\sigma(\widetilde W_j),
		\]
		where \(\vee\) denotes the smallest sigma-field containing both sigma-fields,
		and \(\sigma(\widetilde W_j)\) is the sigma-field generated by the auxiliary
		white noise \(\widetilde W_j\). Then \(Z_j\), \(\sigma_j\), and \(\sigma_i\) are
		\(\mathcal G_{i,j}\)-measurable, while \(\widetilde\Delta_i(t_{i-1},\delta)\) is
		independent of \(\mathcal G_{i,j}\). Indeed, the part of \(\widetilde\Delta_i(t_{i-1},\delta)\) driven by
		\(\widetilde W_i\) is independent of both \(W\) and \(\widetilde W_j\),
		whereas its part driven by \(W\) only involves the noise after time
		\(t_{i-1}(\delta)\), which is independent of \(\mathcal F_{t_{i-1}(\delta)}\).
		Since \(\widetilde W_i\) and \(\widetilde W_j\) are independent for
		\(i\ne j\), it follows that
		\(\widetilde\Delta_i(t_{i-1},\delta)\) is independent of
		\(\mathcal G_{i,j}\).
		Hence
		\[
		\begin{aligned}
			\mathbb E[\sigma_i^4\sigma_j^4Z_iZ_j]
			&=
			\mathbb E\left[
			\sigma_i^4\sigma_j^4Z_j
			\,\mathbb E(Z_i\mid\mathcal G_{i,j})
			\right].
		\end{aligned}
		\]
		Since \(Z_i\) is independent of \(\mathcal G_{i,j}\), we have
		\[
		\mathbb E(Z_i\mid\mathcal G_{i,j})
		=
		\mathbb E Z_i
		=
		\mathbb E[\widetilde\Delta_i^4(t_{i-1},\delta)]-c_x\delta.
		\]
		Hence,
		\[
		\left|
		\mathbb E(Z_i\mid\mathcal G_{i,j})
		\right|
		\le C\delta^{5/2}.
		\]
		Consequently, by the Cauchy--Schwarz inequality and the uniform moment bounds,
		\[
		\mathbb E|\sigma_i^4\sigma_j^4Z_j|
		\le
		\big(\mathbb E|\sigma_i|^8|\sigma_j|^8\big)^{1/2}
		\big(\mathbb E|Z_j|^2\big)^{1/2} \le C_{T_1,T_2}\delta.
		\]
		Therefore,
\[
	\sum_{\substack{1\leq j<i\leq n\\i-j\geq \left\lfloor\delta^{-1/5}\right\rfloor+2}}
	\left|
	\mathbb E[\sigma_i^4\sigma_j^4Z_iZ_j]
	\right|
	\leq
	C_{T_1,T_2,x}
	n^2\delta^{7/2}
	\longrightarrow0.
\]
The case \(j-i\geq \left\lfloor\delta^{-1/5}\right\rfloor+2\) is treated in the same way. Combining the diagonal, near-diagonal, and far-off-diagonal estimates, we
obtain
\[
\begin{aligned}
	\mathbb E|J_2|^2
	&=
	\sum_{i,j=1}^{n}
	\mathbb E\bigl[\sigma_i^4\sigma_j^4Z_iZ_j\bigr]\\
	&\leq
	C_{T_1,T_2,x}
	\left(
	n\delta^2
	+n(\left\lfloor\delta^{-1/5}\right\rfloor+2)\delta^2
	+n^2\delta^{7/2}
	\right)\\
	&\leq
	C_{T_1,T_2,x}
	\left(
	n^{-1}
	+n^{-4/5}
	+n^{-3/2}
	\right)\\
	&\leq
	C_{T_1,T_2,x}n^{-4/5}.
\end{aligned}
\]
		Consequently,
		\[
		J_2\longrightarrow0
		\qquad\text{in }L^2(\Omega).
		\]
		Moreover, by the Cauchy--Schwarz inequality,
		\begin{equation}\label{eq:J_2-L1}
					\mathbb E|J_2|
			\le
			\big(\mathbb E|J_2|^2\big)^{1/2}
			\le
			C_{T_1,T_2}n^{-2/5}.
		\end{equation}
		
		Therefore, \eqref{eq:L1con} follows from \eqref{eq:J_3-L1}, \eqref{eq:J_1-L1} and \eqref{eq:J_2-L1}, which completes the proof.
\end{proof}

As a consequence of Theorem \ref{th:tem vari}, we obtain a consistent
estimator of the diffusion coefficient defined in \eqref{eq:A(x)} at any fixed spatial point
away from the interface.
\begin{remark}
	Assume that \(\sigma\) is known. Fix
\(x\in\mathbb R\setminus\{0\}\), and suppose additionally that
\begin{equation}\label{eq:assum}
		\mathbb P\left(
		\int_{T_1}^{T_2}\sigma^4(u(r,x))\,dr>0
		\right)=1.
	\end{equation}
    Note that this condition is automatically satisfied under Assumption
\ref{ass:nondegenerate}, since
\[
\int_{T_1}^{T_2}\sigma^4(u(r,x))\,dr
\geq
\sigma_0^4(T_2-T_1)>0
\]
almost surely.
Since \(\tau(x)>0\), for \(x\neq0\), define
\[
\widehat A_n(x)
:=
\begin{cases}
\displaystyle
\frac{
6(T_2-T_1)\sum_{i=1}^n\sigma^4(u(t_i,x))
}{
n\pi\sum_{i=1}^n
\bigl(u(t_i,x)-u(t_{i-1},x)\bigr)^4
},
&
V_{n,x}(u)>0,
\\[4mm]
0,
&
V_{n,x}(u)=0.
\end{cases}
\]The estimator is based on high-frequency temporal observations of \(u\) at a fixed spatial point over the fixed interval \([T_1,T_2]\).

By the Riemann-sum argument,
\[
\delta\sum_{i=1}^{n}\sigma^4(u(t_i,x))
\longrightarrow
\int_{T_1}^{T_2}\sigma^4(u(r,x))\,dr
\]
in probability. Therefore, Theorem~\ref{th:tem vari}, the assumption \eqref{eq:assum}, and the continuous mapping theorem imply that
\[
\widehat A_n(x)\xrightarrow{\mathbb P}A(x).
\]

In particular, \(\widehat A_n(x)\) is a consistent estimator of
\(a_1\) when \(x<0\), and of \(a_2\) when \(x>0\).

At the interface \(x=0\), the factor
\(\tau(0)=\eta^2\) depends on the model parameters entering
\(\eta\). Consequently, estimation at the interface requires
additional information or a joint estimation procedure and is not
considered here.
\end{remark}

Motivated by Hildebrandt and Trabs~\cite{MR4280160}, we introduce the
spatial average of the temporal quartic variations over an interval
away from the interface. We now justify the uniformity in the spatial variable. Fix \(\varepsilon\in(0,1)\). An inspection
of the proof of the linear variance expansion in
\cite[Lemma~5]{MR3992994} shows that its \(O(\delta^2)\) remainder
can be chosen uniformly for
\[
(s,x)\in[T_1,T_2]\times[\varepsilon,1].
\]
Indeed, on this set, \(A(x)=a_2\) and \(\operatorname{sign}(x)=1\),
while the functions and derivatives involved in the corresponding
Taylor expansions remain uniformly bounded, since \(T_1>0\) and
\(\varepsilon>0\). Together with the uniform moment and increment
estimates established above and Lemma~\ref{le:imp esti}, this shows that all the
estimates in the proof of Theorem~\ref{th:tem vari} are uniform for
\(x\in[\varepsilon,1]\). Consequently, \eqref{eq:L1con} holds uniformly
on \([\varepsilon,1]\).

\begin{corollary}
	For \(n,m\geq1\), let
	\[
	h_m:=\frac{1-\varepsilon}{m},
	\qquad
	x_j:=\varepsilon+jh_m,
	\qquad j=0,1,\ldots,m,
	\]
	and define
	\[
	V_{n,m}^{\varepsilon}(u)
	:=
	h_m
	\sum_{j=0}^{m-1}
	\sum_{i=1}^{n}
	\left(
	u(t_i,x_j)-u(t_{i-1},x_j)
	\right)^4
	=
	h_m
	\sum_{j=0}^{m-1}V_{n,x_j}(u),
	\]
	where \(V_{n,x}(u)\) is defined in \eqref{eq:V} and
	\(t_i\) is given by \eqref{eq:t_i}. Then, as \(n,m\to\infty\),
	\begin{equation}\label{eq:Vnmright}
		V_{n,m}^{\varepsilon}(u)
		\longrightarrow
		\frac{6}{\pi a_2}
		\int_{\varepsilon}^{1}
		\int_{T_1}^{T_2}
		\sigma^4(u(t,y))\,dt\,dy
		\qquad\text{in }L^1(\Omega).
	\end{equation}
	Moreover,
	\begin{equation}
		\begin{aligned}
			&\mathbb E\left|
			V_{n,m}^{\varepsilon}(u)
			-
			\frac{6}{\pi a_2}
			\int_{\varepsilon}^{1}
			\int_{T_1}^{T_2}
			\sigma^4(u(t,y))\,dt\,dy
			\right|\\
			&\qquad\leq
			C_{T_1,T_2,\varepsilon}
			\left(
			n^{-\frac{3}{20}}
			+
			m^{-\frac{1}{2}}
			\right)
		\end{aligned}
	\end{equation}
	for \(n\) and \(m\) sufficiently large.
\end{corollary}

\begin{proof}
	For simplicity, set
	\[
	\Phi(y):=
	\frac{6}{\pi a_2}
	\int_{T_1}^{T_2}\sigma^4(u(t,y))\,dt.
	\]
	We decompose
	\begin{align*}
		&V_{n,m}^{\varepsilon}(u)
		-\int_{\varepsilon}^{1}\Phi(y)\,dy\\
		&=
		V_{n,m}^{\varepsilon}(u)
		-
		h_m\sum_{j=0}^{m-1}
		\frac{6}{\pi a_2}
		\int_{T_1}^{T_2}
		\sigma^4(u(t,x_j))\,dt\\
		&\quad+
		h_m\sum_{j=0}^{m-1}
		\frac{6}{\pi a_2}
		\int_{T_1}^{T_2}
		\sigma^4(u(t,x_j))\,dt
		-
		\int_{\varepsilon}^{1}\Phi(y)\,dy\\
		&=:D_1+D_2.
	\end{align*}
	
	Since the estimates in the proof of Theorem~\ref{th:tem vari} are
	uniform for \(x\in[\varepsilon,1]\), by \eqref{eq:L1con} and the
	triangle inequality, we obtain
	\begin{align*}
		\mathbb E|D_1|
		&\leq
		h_m\sum_{j=0}^{m-1}
		\mathbb E\left|
		V_{n,x_j}(u)
		-
		\frac{6}{\pi a_2}
		\int_{T_1}^{T_2}
		\sigma^4(u(t,x_j))\,dt
		\right|\\
		&\leq
		C_{T_1,T_2,\varepsilon}
		h_m m n^{-\frac{3}{20}}\\
		&\leq
		C_{T_1,T_2,\varepsilon}
		n^{-\frac{3}{20}}.
	\end{align*}
	
	For the second term, we have
	\[
	D_2
	=
	\sum_{j=0}^{m-1}
	\int_{x_j}^{x_{j+1}}
	\left(
	\Phi(x_j)-\Phi(y)
	\right)\,dy.
	\]
	Therefore,
	\begin{align*}
		\mathbb E|D_2|
		&\leq
		\frac{6}{\pi a_2}
		\sum_{j=0}^{m-1}
		\int_{x_j}^{x_{j+1}}
		\int_{T_1}^{T_2}
		\mathbb E\left|
		\sigma^4(u(t,x_j))
		-
		\sigma^4(u(t,y))
		\right|
		\,dt\,dy.
	\end{align*}
	
	By the Lipschitz continuity of \(\sigma\), the uniform moment bounds of
	the mild solution, and Lemma~\ref{le:helix-x}, we have
	\[
	\mathbb E\left|
	\sigma^4(u(t,x_j))
	-
	\sigma^4(u(t,y))
	\right|
	\leq
	C_{T_1,T_2}
	\left(
	\mathbb E|u(t,x_j)-u(t,y)|^4
	\right)^{1/4}
	\leq
	C_{T_1,T_2}
	|x_j-y|^{1/2}.
	\]
	Consequently,
	\begin{align*}
		\mathbb E|D_2|
		&\leq
		C_{T_1,T_2}
		\sum_{j=0}^{m-1}
		\int_{x_j}^{x_{j+1}}
		|x_j-y|^{1/2}\,dy\\
		&=
		C_{T_1,T_2}
		m\int_0^{h_m}y^{1/2}\,dy\\
		&\leq
		C_{T_1,T_2,\varepsilon}
		m^{-1/2}.
	\end{align*}
	
	Combining the estimates for \(D_1\) and \(D_2\), we obtain
	\[
	\mathbb E\left|
	V_{n,m}^{\varepsilon}(u)
	-
	\frac{6}{\pi a_2}
	\int_{\varepsilon}^{1}
	\int_{T_1}^{T_2}
	\sigma^4(u(t,y))\,dt\,dy
	\right|
	\leq
	C_{T_1,T_2,\varepsilon}
	\left(
	n^{-\frac{3}{20}}
	+
	m^{-\frac{1}{2}}
	\right),
	\]
	which completes the proof.
\end{proof}
\begin{remark}
	Assume that \(\sigma\) is known and additionally that
	\[
	\mathbb P\left(
	\int_{\varepsilon}^{1}\int_{T_1}^{T_2}
	\sigma^4(u(t,y))\,dt\,dy>0
	\right)=1.
	\]
	We define
	\[
	\widehat a_{2,n,m}^{\,\varepsilon}
	:=
	\frac{
		6h_m\delta
		\sum_{j=0}^{m-1}\sum_{i=1}^{n}
		\sigma^4(u(t_i,x_j))
	}{
		\pi V_{n,m}^{\varepsilon}(u)
	},
	\]
	whenever \(V_{n,m}^{\varepsilon}(u)>0\), and set
	\(\widehat a_{2,n,m}^{\,\varepsilon}=0\) otherwise. 
	
	Then,
	\[
	\widehat a_{2,n,m}^{\,\varepsilon}
	\xrightarrow{\mathbb P}a_2
	\qquad\text{as }n,m\to\infty.
	\]
	
	Indeed, the joint continuity of \(u\), the continuity of \(\sigma\),
	and the Riemann-sum theorem yield
	\[
	h_m\delta
	\sum_{j=0}^{m-1}\sum_{i=1}^{n}
	\sigma^4(u(t_i,x_j))
	\longrightarrow
	\int_{\varepsilon}^{1}\int_{T_1}^{T_2}
	\sigma^4(u(t,y))\,dt\,dy
	\]
	in probability. On the other hand, by \eqref{eq:Vnmright},
	\[
	V_{n,m}^{\varepsilon}(u)
	\longrightarrow
	\frac{6}{\pi a_2}
	\int_{\varepsilon}^{1}\int_{T_1}^{T_2}
	\sigma^4(u(t,y))\,dt\,dy
	\]
	in probability. Therefore, the positivity assumption and the continuous
	mapping theorem imply that
	\[
	\widehat a_{2,n,m}^{\,\varepsilon}
	\xrightarrow{\mathbb P}
	\frac{
		6\displaystyle\int_{\varepsilon}^{1}\int_{T_1}^{T_2}
		\sigma^4(u(t,y))\,dt\,dy
	}{
		\pi\displaystyle\frac{6}{\pi a_2}
		\int_{\varepsilon}^{1}\int_{T_1}^{T_2}
		\sigma^4(u(t,y))\,dt\,dy
	}
	=a_2.
	\]
\end{remark}
\appendix
\section{Proof of the technical lemma}\label{appendix}
\begin{proof}[Proof of Lemma \ref{le:G-G2}]
		To begin with, we write the left-hand side as
	\begin{eqnarray*}
		&&\int_0^s\int_{\mathbb{R}}\left(G(t+h-r,x,y)-G(t-r,x,y)\right)^2dy\,dr\\
		&=&\int_0^s\int_{-\infty}^0\left[\frac{1}{\sqrt{2\pi a_1(t+h-r)}}\left(E^{-}(t+h-r,x,y)-\beta E^{+}(t+h-r,x,y)\right)\right.\\
		&&{}-\left.\frac{1}{\sqrt{2\pi a_1(t-r)}}\left(E^{-}(t-r,x,y)-\beta E^{+}(t-r,x,y)\right)\right]^2dy\,dr\\
		&&{}+\int_0^s\int_{0}^{+\infty}\left[\frac{1}{\sqrt{2\pi a_2(t+h-r)}}\left(E^{-}(t+h-r,x,y)+\beta E^{+}(t+h-r,x,y)\right)\right.\\
		&&{}-\left.\frac{1}{\sqrt{2\pi a_2(t-r)}}\left(E^{-}(t-r,x,y)+\beta E^{+}(t-r,x,y)\right)\right]^2dy\,dr,
	\end{eqnarray*}
	with
	\[E^{-}(t,x,y)=\exp    \left(-\frac{(f(y)-f(x))^{2}}{2t}   \right)\]
	and
	\[E^{+}(t,x,y)=\exp   \left (-\frac{( \lvert f(y) \rvert + \lvert f(x) \rvert )^{2}}{2t}  \right ).\]
	We find that
	\begin{align*}
		&\max\left(\frac{1}{a_1}\left[\frac{E^{-}(t+h-r,x,y)}{\sqrt{ (t+h-r)}}-\frac{\beta E^{+}(t+h-r,x,y)}{\sqrt{ (t+h-r)}}-\frac{E^{-}(t-r,x,y)}{\sqrt{ (t-r)}}+\frac{\beta E^{+}(t-r,x,y)}{\sqrt{ (t-r)}}\right]^2\,\textbf{,}\right.\\
		& \left.\qquad\frac{1}{a_2}\left[\frac{E^{-}(t+h-r,x,y)}{\sqrt{ (t+h-r)}}+\frac{\beta E^{+}(t+h-r,x,y)}{\sqrt{ (t+h-r)}}-\frac{E^{-}(t-r,x,y)}{\sqrt{ (t-r)}}-\frac{\beta E^{+}(t-r,x,y)}{\sqrt{ (t-r)}}\right]^2\right)\\
		&=\max\left(\frac{1}{a_1}\left[\frac{E^{-}(t+h-r,x,y)}{\sqrt{ (t+h-r)}}-\frac{E^{-}(t-r,x,y)}{\sqrt{ (t-r)}}-\beta\left(\frac{ E^{+}(t+h-r,x,y)}{\sqrt{ (t+h-r)}}-\frac{ E^{+}(t-r,x,y)}{\sqrt{ (t-r)}}\right)\right]^2\,\textbf{,}\right.\\
		&   \left.\qquad\frac{1}{a_2}\left[\frac{E^{-}(t+h-r,x,y)}{\sqrt{ (t+h-r)}}-\frac{E^{-}(t-r,x,y)}{\sqrt{ (t-r)}}+\beta\left(\frac{ E^{+}(t+h-r,x,y)}{\sqrt{ (t+h-r)}}-\frac{ E^{+}(t-r,x,y)}{\sqrt{ (t-r)}}\right)\right]^2\right)\\
		&\leq  2\max\left(\frac{1}{a_1}, \frac{1}{a_2}\right)\left[\left(\frac{E^{-}(t+h-r,x,y)}{\sqrt{ (t+h-r)}}-\frac{E^{-}(t-r,x,y)}{\sqrt{ (t-r)}}\right)^2\right.\\
		&\qquad\qquad\qquad\qquad\qquad\qquad\left.+\,\beta^2\left(\frac{ E^{+}(t+h-r,x,y)}{\sqrt{ (t+h-r)}}-\frac{ E^{+}(t-r,x,y)}{\sqrt{ (t-r)}}\right)^2\right],
	\end{align*}
	since 
	\[\max\left(\frac{1}{a_1}(a-b)^2,\, \frac{1}{a_2}(a+b)^2\right)\,\leq\, 2\max \left(\frac{1}{a_1},\, \frac{1}{a_2}\right)\cdot\left(a^2\,+b^2\right)\]
	with basic calculus for every $a,b\in\mathbb{R}$.
	Hence we obtain
	\begin{eqnarray*}
		&&\int_0^s\int_{\mathbb{R}}\left(G(t+h-r,x,y)-G(t-r,x,y)\right)^2dy\,dr\\
		&\leq & \frac{1}{\pi} \max\left(\frac{1}{a_1}, \frac{1}{a_2}\right) \left[ \int_0^s\int_{\mathbb{R}}\left(\frac{E^{-}(t+h-r,x,y)}{\sqrt{ (t+h-r)}}-\frac{E^{-}(t-r,x,y)}{\sqrt{ (t-r)}}\right)^2\,dy\,dr\right.\\
		&&{}+   \left.\beta^2\int_0^s\int_{\mathbb{R}} \left(\frac{ E^{+}(t+h-r,x,y)}{\sqrt{ (t+h-r)}}-\frac{ E^{+}(t-r,x,y)}{\sqrt{ (t-r)}}\right)^2\,dy\,dr  \right].
	\end{eqnarray*} 
	It remains to estimate the two integral given in the right-hand side and we proceed with the proof in two steps.
	
	\textit{Step \I :}
	
	Rewrite the integral
	\[\int_0^s\int_{\mathbb{R}}\left(\frac{E^{-}(t+h-r,x,y)}{\sqrt{ (t+h-r)}}-\frac{E^{-}(t-r,x,y)}{\sqrt{ (t-r)}}\right)^2\,dy\,dr\]
	we get
	\begin{eqnarray*}
		&&\int_0^s\int_{\mathbb{R}}\left(\frac{1}{\sqrt{ (t+h-r)}}\exp \left(-\frac{(f(y)-f(x))^{2}}{2(t+h-r)}\right) \right.\\
		&&{}-\left.\frac{1}{\sqrt{ (t-r)}}\exp \left(-\frac{(f(y)-f(x))^{2}}{2(t-r)}\right)\right)^2\,dy\,dr\\
		&=& \int_0^s\int_{0}^{+\infty}\left(\frac{1}{\sqrt{ (t+h-r)}}\exp \left(-\frac{(y/\sqrt{a_2}-f(x))^{2}}{2(t+h-r)}\right) \right.\\
		&&{}\left.-\frac{1}{\sqrt{ (t-r)}}\exp \left(-\frac{(y/\sqrt{a_2}-f(x))^{2}}{2(t-r)}\right)\right)^2\,dy\,dr\\
		&&{}+\int_0^s\int_{-\infty}^0\left(\frac{1}{\sqrt{ (t+h-r)}}\exp \left(-\frac{(y/\sqrt{a_1}-f(x))^{2}}{2(t+h-r)}\right) \right.\\
		&&{}\left.-\frac{1}{\sqrt{ (t-r)}}\exp \left(-\frac{(y/\sqrt{a_1}-f(x))^{2}}{2(t-r)}\right)\right)^2\,dy\,dr.
	\end{eqnarray*}
	
	By a change of variables $y/\sqrt{a_2}-f(x)=y'$ in the first integral and $y/\sqrt{a_1}-f(x)=y'$ in the second integral, it follows that
	\begin{align*}
		\MoveEqLeft[1.5]\int_0^s\int_{\mathbb{R}}\left(\frac{E^{-}(t+h-r,x,y)}{\sqrt{ (t+h-r)}}-\frac{E^{-}(t-r,x,y)}{\sqrt{ (t-r)}}\right)^2\,dy\,dr\\
		&=\sqrt{a_2}\int_0^s\int_{-f(x)}^{+\infty}\left(\frac{1}{\sqrt{ (t+h-r)}}\exp \left(-\frac{y^{2}}{2(t+h-r)}\right) -\frac{1}{\sqrt{ (t-r)}}\exp \left(-\frac{y^{2}}{2(t-r)}\right)\right)^2\,dy\,dr\\
		&+\sqrt{a_1}\int_0^s\int_{-\infty}^{-f(x)}\left(\frac{1}{\sqrt{ (t+h-r)}}\exp \left(-\frac{y^{2}}{2(t+h-r)}\right) -\frac{1}{\sqrt{ (t-r)}}\exp \left(-\frac{y^{2}}{2(t-r)}\right)\right)^2\,dy\,dr\\
		&\leq 2\pi \max(\sqrt{a_1}, \sqrt{a_2})\int_0^s\int_{\mathbb{R}}\left(\frac{1}{\sqrt{ 2\pi(t+h-r)}}\exp \left(-\frac{y^{2}}{2(t+h-r)}\right) \right.\\
		&\qquad\qquad\qquad\qquad\qquad\qquad\left.-\frac{1}{\sqrt{2\pi (t-r)}}\exp \left(-\frac{y^{2}}{2(t-r)}\right)\right)^2\,dy\,dr\\
		&\,:\,=\,2\pi \max(\sqrt{a_1}, \sqrt{a_2})\int_0^s\int_{\mathbb{R}}(p_{t+h-r}(y)-p_{t-r}(y))^2\,dy\,dr,
	\end{align*}
	where $p_t(x)$ denotes the standard heat kernel defined by
	\[
	p_t(x) = \frac{1}{\sqrt{2\pi t}}
	\exp\!\left( -\,\frac{x^{2}}{2t} \right),
	\]
	for every \(t>0\) and \(x\in\mathbb{R}\).
	
	Using the Parseval--Plancherel identity
	\[
	\langle f, g \rangle_{L^{2}(\mathbb{R})}
	= \frac{1}{2\pi} \int_{\mathbb{R}}
	(\mathcal{F}f)(\xi)\, \overline{(\mathcal{F}g)(\xi)}\, d\xi,
	\]
	for any \( f\,,\, g \in L^{2}(\mathbb{R}) \),	and the fact that
	\[
	\mathcal{F}p_t(\cdot)(\xi)=e^{-t\xi^{2}/2},
	\qquad \xi\in\mathbb{R},\ t>0,
	\]
	we derive that
	\begin{eqnarray*}
		&&\int_0^s\int_{\mathbb{R}}\left(\frac{E^{-}(t+h-r,x,y)}{\sqrt{ (t+h-r)}}-\frac{E^{-}(t-r,x,y)}{\sqrt{ (t-r)}}\right)^2\,dy\,dr\\
		&\leq& \max(\sqrt{a_1}, \sqrt{a_2})\int_0^s\int_{\mathbb{R}}\left(e^{-\frac{(t+h-r)\xi^2}{2}}-e^{-\frac{(t-r)\xi^2}{2}}\right)^2\,d\xi\,dr\\
		&=&\max(\sqrt{a_1}, \sqrt{a_2})\int_0^s\int_{\mathbb{R}}e^{-(t-r)\xi^2}\left(1-e^{-\frac{h\xi^2}{2}}\right)^2\,d\xi\,dr\\
		&\leq &\max(\sqrt{a_1}, \sqrt{a_2})\int_{\mathbb{R}}\frac{\left(1-e^{-h\xi^2/2}\right)^2}{\xi^2}\,e^{-(t-s)\xi^2}\left(1-e^{-s\xi^2}\right)\,d\xi.
	\end{eqnarray*}
	Moreover, it is clear that $1-e^{-x}\leq x\wedge 1$ for $x>0$, so we obtain that
	\begin{multline*}
		\int_0^s\int_{\mathbb{R}}\left(\frac{E^{-}(t+h-r,x,y)}{\sqrt{ (t+h-r)}}-\frac{E^{-}(t-r,x,y)}{\sqrt{ (t-r)}}\right)^2\,dy\,dr\\
		\leq2\max(\sqrt{a_1}, \sqrt{a_2})\int_0^\infty\frac{h^2\xi^2}{4}\,e^{-(t-s)\xi^2}\,d\xi\leq \frac{1}{2}\max(\sqrt{a_1}, \sqrt{a_2})\Gamma\left(\frac{3}{2}\right)\,(t-s)^{-\frac{3}{2}}\,h^2,
	\end{multline*}
	with the Gamma function $\Gamma(z)=\int_0^\infty e^{-t}\, t^{z-1}dt\,,\, z>0.$
	
	\textit{Step \II: }
	
	Let us deal with the second integral of the form
	\[\int_0^s\int_{\mathbb{R}} \left(\frac{ E^{+}(t+h-r,x,y)}{\sqrt{ (t+h-r)}}-\frac{ E^{+}(t-r,x,y)}{\sqrt{ (t-r)}}\right)^2\,dy\,dr.\]
	Indeed, we split the integral over $y$ into the positive and negative parts as 
	\begin{multline*}
		\int_0^s\int_{0}^{+\infty}\left(\frac{1}{\sqrt{t+h-r}}\exp \left(-\frac{(\vert f(x)\vert+y/\sqrt{a_2})^2}{2(t+h-r)}\right)-\frac{1}{\sqrt{t-r}}\exp\left(-\frac{(|f(x)|+y/\sqrt{a_2})^2}{2(t-r)}\right)\right)^2\,dy\,dr\\
		+\int_0^s\int_{-\infty}^0\left(\frac{1}{\sqrt{t+h-r}}\exp \left(-\frac{(|f(x)|-y/\sqrt{a_1})^2}{2(t+h-r)}\right)-\frac{1}{\sqrt{t-r}}\exp\left(-\frac{(| f(x)|-y/\sqrt{a_1})^2}{2(t-r)}\right)\right)^2\,dy\,dr.
	\end{multline*}
	Then, by a change of variables $\vert f(x)\vert+y/\sqrt{a_2}=y'$ in the first integral and $\vert f(x)\vert-y/\sqrt{a_1}=y'$ in the second integral, we have
	\begin{multline*}
		\int_0^s\int_{\mathbb{R}} \left(\frac{ E^{+}(t+h-r,x,y)}{\sqrt{ (t+h-r)}}-\frac{ E^{+}(t-r,x,y)}{\sqrt{ (t-r)}}\right)^2\,dy\,dr\\
		\leq \max(\sqrt{a_1}, \sqrt{a_2}) \int_0^s\int_{\mathbb{R}}\left(\frac{1}{\sqrt{(t+h-r)}}\exp \left(-\frac{y^{2}}{2(t+h-r)}\right)-\frac{1}{\sqrt{(t-r)}}\exp \left(-\frac{y^{2}}{2(t-r)}\right)\right)^2\,dy\,dr.
	\end{multline*}
	
	The same Fourier estimate applies to this term, completing the proof.
\end{proof}

\bibliographystyle{plain}
\bibliography{Li_exact_variation}

@article{MR3992994,
	author  = {Zili, M. and Zougar, E.},
	title   = {Exact variations for stochastic heat equations with piecewise constant coefficients and application to parameter estimation},
	journal = {Theory Probab. Math. Statist.},
	number  = {100},
	pages   = {75--101},
	year    = {2019},
	doi     = {10.1090/tpms/1099},
	url     = {https://doi.org/10.1090/tpms/1099},
	mrnumber = {3992994}
}

@article{MR4028081,
	author  = {Zili, M. and Zougar, E.},
	title   = {Spatial quadratic variations for the solution to a stochastic partial differential equation with elliptic divergence form operator},
	journal = {Mod. Stoch. Theory Appl.},
	volume  = {6},
	number  = {3},
	pages   = {345--375},
	year    = {2019},
	doi     = {10.15559/19-VMSTA139},
	url     = {https://doi.org/10.15559/19-VMSTA139},
	mrnumber = {4028081}
}

@article{MR2321899,
	author  = {Posp\'i\v{s}il, J. and Tribe, R.},
	title   = {Parameter estimates and exact variations for stochastic heat equations driven by space-time white noise},
	journal = {Stoch. Anal. Appl.},
	volume  = {25},
	number  = {3},
	pages   = {593--611},
	year    = {2007},
	doi     = {10.1080/07362990701282849},
	url     = {https://doi.org/10.1080/07362990701282849},
	mrnumber = {2321899}
}

@article{MR4192903,
	author  = {Mishura, Y. and Ralchenko, K. and Zili, M. and Zougar, E.},
	title   = {Fractional stochastic heat equation with piecewise constant coefficients},
	journal = {Stoch. Dyn.},
	volume  = {21},
	number  = {1},
	pages   = {Paper No. 2150002, 39 pp.},
	year    = {2021},
	doi     = {10.1142/S0219493721500027},
	url     = {https://doi.org/10.1142/S0219493721500027},
	mrnumber = {4192903}
}

@article{MR4421358,
	author  = {Zili, M. and Zougar, E.},
	title   = {Stochastic heat equation with piecewise constant coefficients and generalized fractional type noise},
	journal = {Theory Probab. Math. Statist.},
	number  = {104},
	pages   = {123--144},
	year    = {2021},
	doi     = {10.1090/tpms/1150},
	url     = {https://doi.org/10.1090/tpms/1150},
	mrnumber = {4421358}
}

@article{MR3988829,
	author  = {Zili, M. and Zougar, E.},
	title   = {One-dimensional stochastic heat equation with discontinuous conductance},
	journal = {Appl. Anal.},
	volume  = {98},
	number  = {12},
	pages   = {2178--2191},
	year    = {2019},
	doi     = {10.1080/00036811.2018.1451642},
	url     = {https://doi.org/10.1080/00036811.2018.1451642},
	mrnumber = {3988829}
}

@article{MR3296333,
	author  = {Chen, Z.-Q. and Zili, M.},
	title   = {One-dimensional heat equation with discontinuous conductance},
	journal = {Sci. China Math.},
	volume  = {58},
	number  = {1},
	pages   = {97--108},
	year    = {2015},
	doi     = {10.1007/s11425-014-4912-1},
	url     = {https://doi.org/10.1007/s11425-014-4912-1},
	mrnumber = {3296333}
}

@article{MR1684157,
	author  = {Dalang, R. C.},
	title   = {Extending the martingale measure stochastic integral with applications to spatially homogeneous s.p.d.e.'s},
	journal = {Electron. J. Probab.},
	volume  = {4},
	number  = {6},
	pages   = {1--29},
	year    = {1999},
	doi     = {10.1214/EJP.v4-43},
	url     = {https://doi.org/10.1214/EJP.v4-43},
	mrnumber = {1684157}
}

@incollection{MR876085,
	author    = {Walsh, J. B.},
	title     = {An introduction to stochastic partial differential equations},
	booktitle = {\'Ecole d'\'et\'e de probabilit\'es de Saint-Flour, XIV---1984},
	series    = {Lecture Notes in Mathematics},
	volume    = {1180},
	pages     = {265--439},
	publisher = {Springer},
	address   = {Berlin},
	year      = {1986},
	doi       = {10.1007/BFb0074920},
	url       = {https://doi.org/10.1007/BFb0074920},
	mrnumber  = {876085}
}

@book{MR4544909,
	author    = {Tudor, C. A.},
	title     = {Stochastic Partial Differential Equations with Additive {G}aussian Noise: Analysis and Inference},
	publisher = {World Scientific},
	address   = {Hackensack, NJ},
	year      = {2023},
	pages     = {xi+192},
	doi       = {10.1142/13089},
	url       = {https://doi.org/10.1142/13089},
	mrnumber  = {4544909}
}

@article{MR2105066,
	author  = {Lejay, A.},
	title   = {Monte {C}arlo methods for fissured porous media: a gridless approach},
	journal = {Monte Carlo Methods Appl.},
	volume  = {10},
	number  = {3--4},
	pages   = {385--392},
	year    = {2004},
	doi     = {10.1515/mcma.2004.10.3-4.385},
	url     = {https://doi.org/10.1515/mcma.2004.10.3-4.385},
	mrnumber = {2105066}
}

@article{MR2217816,
	author  = {\'Etor\'e, P.},
	title   = {On random walk simulation of one-dimensional diffusion processes with discontinuous coefficients},
	journal = {Electron. J. Probab.},
	volume  = {11},
	number  = {9},
	pages   = {249--275},
	year    = {2006},
	doi     = {10.1214/EJP.v11-311},
	url     = {https://doi.org/10.1214/EJP.v11-311},
	mrnumber = {2217816}
}

@article{MR917463,
	author  = {Gaveau, B. and Okada, M. and Okada, T.},
	title   = {Second order differential operators and {D}irichlet integrals with singular coefficients. {I}. Functional calculus of one-dimensional operators},
	journal = {Tohoku Math. J. (2)},
	volume  = {39},
	number  = {4},
	pages   = {465--504},
	year    = {1987},
	doi     = {10.2748/tmj/1178228238},
	url     = {https://doi.org/10.2748/tmj/1178228238},
	mrnumber = {917463}
}

@article{CANTRELL1999189,
	author  = {Cantrell, R. S. and Cosner, C.},
	title   = {Diffusion Models for Population Dynamics Incorporating Individual Behavior at Boundaries: Applications to Refuge Design},
	journal = {Theor. Popul. Biol.},
	volume  = {55},
	number  = {2},
	pages   = {189--207},
	year    = {1999},
	doi     = {10.1006/tpbi.1998.1397},
	url     = {https://doi.org/10.1006/tpbi.1998.1397}
}

@incollection{MR839024,
	author    = {Nicaise, S.},
	title     = {Some results on spectral theory over networks, applied to nerve impulse transmission},
	booktitle = {Orthogonal Polynomials and Applications (Bar-le-Duc, 1984)},
	series    = {Lecture Notes in Mathematics},
	volume    = {1171},
	pages     = {532--541},
	publisher = {Springer},
	address   = {Berlin},
	year      = {1985},
	doi       = {10.1007/BFb0076584},
	url       = {https://doi.org/10.1007/BFb0076584},
	mrnumber  = {839024}
}

@book{MR1500166,
	author    = {Dalang, R. C. and Khoshnevisan, D. and Mueller, C. and Nualart, D. and Xiao, Y.},
	editor    = {Khoshnevisan, D. and Rassoul-Agha, F.},
	title     = {A Minicourse on Stochastic Partial Differential Equations},
	series    = {Lecture Notes in Mathematics},
	volume    = {1962},
	publisher = {Springer},
	address   = {Berlin, Heidelberg},
	year      = {2009},
	pages     = {xi+222},
	doi       = {10.1007/978-3-540-85994-9},
	url       = {https://doi.org/10.1007/978-3-540-85994-9},
	mrnumber  = {1500166}
}

@article{MR4280160,
	author  = {Hildebrandt, F. and Trabs, M.},
	title   = {Parameter estimation for {SPDE}s based on discrete observations in time and space},
	journal = {Electron. J. Stat.},
	volume  = {15},
	number  = {1},
	pages   = {2716--2776},
	year    = {2021},
	doi     = {10.1214/21-EJS1848},
	url     = {https://doi.org/10.1214/21-EJS1848},
	mrnumber = {4280160}
}

@misc{arXiv:2502.15351,
	author        = {Agram, N. and Turpin, I. and Zougar, E.},
	title         = {Spatially Controlled Evolution of Composite Materials via Stochastic Partial Differential Equations},
	year          = {2025},
	eprint        = {2502.15351},
	archiveprefix = {arXiv},
	primaryclass  = {math.OC},
	url           = {https://arxiv.org/abs/2502.15351}
}

@article{li2025spatial,
	author  = {Li, Y. and Shu, H. and Yan, L.},
	title   = {Spatial Quadratic Variation for Stochastic Heat Equations Driven by Multiplicative Noise with Piecewise Constant Coefficients},
	journal = {Frontiers of Mathematics},
	pages   = {1--18},
	year    = {2025},
	doi     = {10.1007/s11464-025-0050-z},
	url     = {https://doi.org/10.1007/s11464-025-0050-z}
}

@article{olivera2025temporal,
	author  = {Olivera, C. and Tudor, C. A.},
	title   = {Temporal quadratic and higher order variation for the nonlinear stochastic heat equation and applications to parameter estimation},
	journal = {Ann. Mat. Pura Appl.},
	volume  = {204},
	number  = {6},
	pages   = {2603--2631},
	year    = {2025},
	doi     = {10.1007/s10231-025-01584-x},
	url     = {https://doi.org/10.1007/s10231-025-01584-x}
}

@article{gamain2023exact,
	author  = {Gamain, J. and Tudor, C. A.},
	title   = {Exact variation and drift parameter estimation for the nonlinear fractional stochastic heat equation},
	journal = {Jpn. J. Stat. Data Sci.},
	volume  = {6},
	number  = {1},
	pages   = {381--406},
	year    = {2023},
	doi     = {10.1007/s42081-023-00188-0},
	url     = {https://doi.org/10.1007/s42081-023-00188-0}
}

@article{li2026exacttime,
	author  = {Li, Y. and Shu, H. and Yan, L.},
	title   = {Exact temporal variation for fractional stochastic heat equation driven by space-time white noise},
	journal = {Commun. Stat. Theory Methods},
	volume  = {55},
	number  = {12},
	pages   = {3704--3718},
	year    = {2026},
	doi     = {10.1080/03610926.2025.2584611},
	url     = {https://doi.org/10.1080/03610926.2025.2584611}
}

@misc{li2026exactspace,
	author       = {Li, Y. and Hu, Y. and Shu, H. and Yan, L.},
	title        = {Exact Variation for Fractional Stochastic Heat Equation and Application to Parameter Estimation},
	howpublished = {SSRN},
	year         = {2026},
	doi          = {10.2139/ssrn.6933428},
	url          = {https://doi.org/10.2139/ssrn.6933428}
}
\end{document}